\documentclass[11pt]{amsart}
\usepackage{amsmath}
\usepackage{amssymb}
\usepackage{amscd}

\def\NZQ{\mathbb}               
\def\NN{{\NZQ N}}
\def\QQ{{\NZQ Q}}
\def\ZZ{{\NZQ Z}}
\def\RR{{\NZQ R}}

\def\PP{{\NZQ P}}

\newtheorem{Theorem}{Theorem}[section]
\newtheorem{Lemma}[Theorem]{Lemma}
\newtheorem{Corollary}[Theorem]{Corollary}
\newtheorem{Proposition}[Theorem]{Proposition}
\newtheorem{Remark}[Theorem]{Remark}

\newtheorem{Example}[Theorem]{Example}

\newtheorem{Definition}[Theorem]{Definition}

\let\epsilon\varepsilon
\let\phi=\varphi
\let\kappa=\varkappa

\begin{document}
\title{Asymptotic multiplicities of graded families of  ideals and linear series}
\author{Steven Dale Cutkosky }
\thanks{Partially supported by NSF}

\address{Steven Dale Cutkosky, Department of Mathematics,
University of Missouri, Columbia, MO 65211, USA}
\email{cutkoskys@missouri.edu}

\begin{abstract} We  find  simple necessary and sufficient conditions on a local ring $R$ of dimension $d$ for  the limit 
$$
\lim_{i\rightarrow\infty}\frac{\ell_R(R/I_n)}{n^d}
$$
to exist whenever $\{I_n\}$ is a 
 graded family of  $m_R$-primary ideals, and
   give a number of applications. We also give simple necessary and sufficient conditions on projective schemes over a field $k$ for asymptotic limits of the growth of all graded linear series of a fixed Kodaira-Iitaka dimension to exist. 
\end{abstract}

\maketitle

\section{Introduction} 
\subsection{Limits of graded families of ideals}
In this paper we prove the following theorem about graded families of $m_R$-primary ideals. 

\begin{Theorem}(Theorem \ref{Theorem4})\label{TheoremI20}  Suppose that $R$ is a (Noetherian) local ring of dimension $d$, and  $N(\hat R)$ is the nilradical of the $m_R$-adic completion $\hat R$ of $R$.  Then   the limit 
\begin{equation}\label{I5}
\lim_{n\rightarrow\infty}\frac{\ell_R(R/I_n)}{n^d}
\end{equation}
exists for any graded family $\{I_n\}$ of $m_R$-primary ideals, if and only if $\dim N(\hat R)<d$.
\end{Theorem}

 A graded family of ideals $\{I_n\}$ in $R$ is a family of ideals indexed by the natural numbers such that $I_0=R$ and $I_mI_n\subset I_{m+n}$ for all $m,n$.

The nilradical $N(R)$ of a $d$-dimensional ring $R$ is 
$$
N(R)=\{x\in R\mid x^n=0 \mbox{ for some positive integer $n$}\}.
$$
Recall that $\dim N(R)=-1$ if $N(R)=0$ and if $N(R)\ne 0$, then
$$
\dim N(R)=\dim R/\mbox{ann}(N(R)),
$$
so that $\dim N(R)=d$ if and only if there exists a minimal prime $P$ of $R$ such that $\dim R/P =d$ and $R_P$ is not reduced.

If $R$ is excellent, then $N(\hat R)= N(R)\hat R$, and the theorem is true with the condition $\dim N(\hat R)<d$ replaced with 
$\dim N(R)<d$. However, there exist Noetherian local domains $R$ (so that $N(R)=0$) such that $\dim N(\hat R)=\dim R$ (Nagata (E3.2) \cite{N}).

It is not difficult to construct examples of graded families of $m_R$-primary ideals in a regular local ring such that the above limit is irrational.

The problem of existence of such limits (\ref{I5}) has been   considered by Ein, Lazarsfeld and Smith \cite{ELS} and Musta\c{t}\u{a} \cite{Mus}.
Lazarsfeld and Musta\c{t}\u{a} \cite{LM} showed that
the limit exists for all graded families of $m_R$-primary ideals in $R$ if $R$ is a domain which is essentially of finite type over an algebraically closed field $k$ with $R/m_R=k$. All of these assumptions are necessary in their proof. Their proof is by reducing the problem to one on graded linear series on a projective variety, and then using a method introduced by Okounkov \cite{Ok} to reduce the problem to one of counting points in an integral semigroup.

In our paper \cite{C1}, we prove that such limits exist for graded families of $m_R$-primary ideals, with the restriction that $R$ 
is an analytically unramified ($N=0$) equicharacteristic  local ring with perfect residue field (Theorem 5.8 \cite{C1}). In this paper we extend this result (in Theorem \ref{Theorem2}) to prove that the limit (\ref{I5}) exists for all graded families of $m_R$-primary ideals  in a local ring $R$ satisfying the assumptions of Theorem \ref{TheoremI20}, establishing  sufficiency in Theorem \ref{TheoremI20}. 
Our proof begins with  the cone method discussed above.

In Example \ref{Example1},
we give an example of a graded family of $m_R$-primary ideals in a nonreduced local ring $R$ for which the above limit does not exist. 
Hailong Dao and   Ilya Smirnov have shown that such examples are universal, so that if the nilradical of 
$R$ has dimension $d$, then there exists a graded family of $m_R$-primary ideals such that the  limit (\ref{I5}) does not exist (Theorem \ref{Theorem3} of this paper). Since a graded family of $m_R$-primary ideals on the completion of a ring lifts to the ring, necessity in Theorem \ref{TheoremI20} follows 
from this result.

In  Section \ref{SecApp} of this paper, we give some applications of this result and the method used in proving it, which generalize some of the applications in \cite{C1}. We extend the theorems to remove the requirement that the local ring be equicharacteristic with perfect residue field, to hold
on arbitrary analytically unramified  local rings.

We prove some volume = multiplicity formulas for graded families of $m_R$-primary ideals in  analytically unramified local rings in Theorems \ref{Theorem15} 
- \ref{Theorem13}. 
Theorem \ref{Theorem15} is proven for valuation ideals associated to an Abhyankar valuation in a regular local ring which is essentially of finite type over a field by Ein, Lazarsfeld and Smith in  \cite{ELS}, for general families of $m_R$-primary ideals when $R$ is a regular local ring containing a field by Musta\c{t}\u{a} in \cite{Mus} and when $R$ is a local domain which is essentially of finite type over an algebraically closed field $k$ with $R/m_R=k$ by Lazarsfeld and Ein in Theorem 3.8 \cite{LM}. All of these assumptions are necessary in the proof in \cite{LM}.
The volume = multiplicity formula is proven when $R$ is regular or $R$ is analytically unramified with perfect residue field in Theorem 6.5 \cite{C1}.

We  give, in Theorems \ref{Theorem14} - Theorem \ref{Theorem15}, some formulas showing that limits of the epsilon multiplicity type exist in analytically unramified local rings. We extend results of \cite{C1}, where it is assumed that $R$ is equicharacteristic with perfect residue field.
Epsilon multiplicity is defined as a limsup by Ulrich and Validashti in \cite{UV} and by Kleiman, Ulrich and Validashti in \cite{KUV}.  We also prove an asymptotic formula on multiplicities proposed by Herzog, Puthenpurakal and Verma \cite{HPV} on analytically unramified local rings. A weaker version of this result is proven in \cite{C1}. A general proof of the existence of epsilon multiplicities for torsion free finite rank modules over an analytically unramified local ring is given in Theorem 3.6 of \cite{C3} by developing the methods of  this paper. In the case of modules over local rings essentially of finite type over an algebraically closed field,  this is proven for modules which are locally free over the punctured spectrum  in \cite{K} and  for more general modules in \cite{C}, using different methods.

\subsection{Kodaira-Iitaka dimension and growth rate of graded linear series}
Before discussing limits of graded linear series we need to define the Kodaira-Iitaka dimension of a graded linear series. This  concept  is defined classically by several equivalent conditions for normal projective varieties. However, these conditions are no longer equivalent for more general proper schemes, so we must make an appropriate choice in our definition.

Suppose that $X$ is a $d$-dimensional proper scheme over a field $k$. A graded linear series $L=\bigoplus_{n\ge 0}L_n$ on $X$ is a graded $k$-subalgebra of a section ring
$\bigoplus_{n\ge 0}\Gamma(X,\mathcal L^n)$ of a line bundle $\mathcal L$ on $X$. We define the Kodaira-Iitaka dimension $\kappa(L)$    from the maximal number of algebraically independent forms in $L$
(the complete definition is given  in Section \ref{SecProp}). This definition agrees with the classical one for normal projective varieties. When $X$ is a projective scheme, the Kodaira-Iitaka dimension of $L$ is $-\infty$ if the Krull dimension of $L$ is 0, and is  one less than the Krull dimension of $L$  if the Krull dimension is positive (Lemma \ref{LKI7}).

Suppose that $\mathcal L$ is  a line bundle  on a normal projective variety $X$.  The index $m(\mathcal L)$ of $\mathcal L$ is defined to be the least common multiple of the positive integers $n$ such that $\Gamma(X,\mathcal L^n)\ne 0$. The theorem of Iitaka,  Theorem 10.2  \cite{I}, tells us that if
 $\kappa(\mathcal L)=-\infty$ then $\dim_k\Gamma(X,\mathcal L^n)=0$ for all positive integers $n$ and   if $\kappa(\mathcal L)\ge 0$ then there exist positive constants $0<a<b$ such that 
 \begin{equation}\label{F1}
 an^{\kappa(L)}<\dim_k\Gamma(X,\mathcal L^{m(\mathcal L)n})<bn^{\kappa(L)}
 \end{equation}
 for $n\gg 0$.
Thus, with the assumption that $X$ is a normal projective variety, the Kodaira-Iitaka dimension is  the growth rate of $\dim_k\Gamma(X,\mathcal L^n)$. Equation (\ref{F1}) continues to hold for graded linear series on a proper variety $X$ over a field (this is stated in (\ref{eq61})). However, when $X$ is not integral, (\ref{F1}) may not hold. In fact, the  rate of growth  has little meaning on nonreduced schemes. In Section \ref{Infinity}, it is shown that 
a graded linear series $L$ on a nonreduced $d$-dimensional projective scheme with
$\kappa(L)=-\infty$ can grow like    $n^d$ (Example \ref{Example2}) or can oscillate wildly between 0 and $n^d$ (Theorem \ref{Theorem21} and Example \ref{Example16}) so that there is no growth rate.

 In fact, it is quite easy to construct badly behaved examples with $\kappa(L)=-\infty$, since in this case the condition that $L_mL_n\subset L_{m+n}$ required for a graded linear series may be  vacuous.

\subsection{Limits of graded families of linear series} Let $k$ be a field.

From (\ref{F1}) and (\ref{eq61}) we have that both $\liminf_{n\rightarrow\infty} \frac{\dim_k L_{mn}}{n^{\kappa(L)}}$ and $\limsup_{n\rightarrow\infty} \frac{\dim_k L_{mn}}{n^{\kappa(L)}}$ exist for a graded linear series $L$ on a proper variety. The remarkable fact is that  they actually exist as a common limit on a proper variety. 

Suppose that $L$ is a graded linear series on a proper variety $X$ over a field $k$. 
The {\it index} $m=m(L)$ of $L$ is defined as the index of groups
$$
m=[\ZZ:G]
$$
where $G$ is the subgroup of $\ZZ$ generated by $\{n\mid L_n\ne 0\}$.

\begin{Theorem}\label{TheoremI1}(Theorem \ref{Theorem5})  Suppose that $X$ is a $d$-dimensional proper variety over a field $k$, and $L$ is a graded linear series on $X$ with Kodaira-Iitaka dimension $\kappa=\kappa(L)\ge 0$. Let $m=m(L)$ be the index of $L$.  Then  
$$
\lim_{n\rightarrow \infty}\frac{\dim_k L_{nm}}{n^{\kappa}}
$$
exists. 
\end{Theorem}

In particular, from the definition of the index, we have that the limit
$$
\lim_{n\rightarrow \infty}\frac{\dim_k L_{n}}{{n}^{\kappa}}
$$
exists, whenever $n$ is constrained to lie in an arithmetic sequence $a+bm$ ($m=m(L)$ and $a$ an arbitrary but fixed constant), as $\dim_kL_n=0$ if $m\not\,\mid n$.

An example of a big line bundle where the limit in Theorem \ref{TheoremI1} is an irrational number is given in Example 4 of Section 7 of the author's paper \cite{CS} with Srinivas.

Theorem \ref{TheoremI1} is proven for big line bundles on a nonsingular variety over an algebraically closed field of characteristic zero by Lazarsfeld (Example 11.4.7 \cite{La}) using Fujita approximation (Fujita, \cite{Fuj}). This result is extended by Takagi using De Jong's theory of alterations \cite{DJ} to hold on nonsingular varieties over algebraically fields of all characteristics $p\ge 0$.
Theorem \ref{TheoremI1} has been proven by  Okounkov  \cite{Ok} for section rings of ample line bundles,  Lazarsfeld and Musta\c{t}\u{a} \cite{LM} for section rings of big line bundles, and for graded linear series by Kaveh and Khovanskii \cite{KK} when $k$ is an algebraically closed field. A local form of this result is given by Fulger in \cite{Ful}. These last proofs use an ingenious method introduced by Okounkov to reduce  to a problem of
counting points in an integral semigroup. All of these proofs require the assumption that {\it $k$ is  algebraically closed}.  In this paper we establish Theorem \ref{TheoremI1} over an arbitrary ground field $k$ (in Theorem \ref{Theorem5}). 
We deduce Fujita approximation over an arbitrary field in Theorem \ref{FApp}.

It is worth remarking that
when $k$ is an arbitrary field and $X$ is geometrically integral over $k$ we easily obtain that the  limit of Theorem \ref{TheoremI1} exists by making the base change to
$\overline X=X\times_k\overline k$ where $\overline k$ is an algebraic closure of $k$. Then $\dim_kL_n=\dim_{\overline k}\overline L_n$
where $\overline L=\bigoplus_{n\ge 0}\overline L_n$ is the graded linear series on $\overline X$ with $\overline L_n=L_n\otimes_k\overline k$.
 The scheme $\overline X$
is a complete $\overline k$ variety (it is integral) since $X$ is geometrically integral. Thus the conclusions of Theorem \ref{TheoremI1}  are valid for $\overline L$ (on the complete variety $\overline X$ over the algebraically closed field $\overline k$)
so that the   limit for $L$ (on $X$ over $k$) exists as well. This observation is exploited by Boucksom and Chen in \cite{BC} where some limits on geometrically integral arithmetic varieties are computed. However, this argument is not applicable when $X$ is not geometrically integral. The most dramatic difficulty can occur when $k$ is not perfect, as there exist simple examples of irreducible
projective varieties which are not even generically reduced after taking the base change to the algebraic closure (we give a simple example below). 
In Theorem \ref{TheoremN1} it is shown that for general graded linear series the limit does not always exist if 
$X$ is not generically reduced.

We now give an example, showing that even if $X$ is normal and $k$ is algebraically closed in the function field of $X$, then 
$X\times_k\overline k$ may not be generically reduced, where $\overline k$ is an algebraic closure of $k$.
Let $p$ be a prime number, $F_p$ be the field with $p$ elements and let $k=F_p(s,t,u)$ be a rational function field in three variables over $F_p$. Let $R$ be the local ring $R=(k[x,y,z]/(sx^p+ty^p+uz^p))_{(x,y,z)}$ with maximal ideal $m_R$.  $R$ is the localization of
$T=F_p[s,t,u,x,y,z]/(sx^p+ty^p+uz^p)$ at the ideal $(x,y,z)$, since $F_p[s,t,u]\cap(x,y,z)=(0)$. $T$ is nonsingular in codimension 1 by the Jacobian criterion over the perfect field $F_p$, and so
$T$ is normal by Serre's criterion. Thus $R$ is normal since it is a localization of $T$. Let $k'$ be the algebraic closure of $k$ in the quotient field $K$ of $R$. Then $k'\subset R$ since $R$ is normal. $R/m_R\cong k$ necessarily contains $k'$, so $k=k'$. However, we have that $R\otimes_k\overline k$ is generically not reduced, if $\overline k$ is an algebraically closure of $k$. Now taking $X$ to be a normal projective model of $K$ over $k$ such that $R$ is the local ring of a closed point of $X$, we have the desired example.
In fact, we have that $k$ is algebraically closed in $K$, but $K\otimes_k\overline k$ has nonzero nilpotent elements.

The statement of Theorem \ref{TheoremI1} generalizes very nicely to reduced proper $k$-schemes, as we establish in Theorem \ref{Theorem18}. 

\begin{Theorem}(Theorem \ref{Theorem18}) Suppose that $X$ is a reduced proper scheme over a  field $k$. 
Let $L$ be a graded linear series on $X$ with Kodaira-Iitaka dimension  $\kappa=\kappa(L)\ge 0$.  
 Then there exists a positive integer $r$ such that 
$$
\lim_{n\rightarrow \infty}\frac{\dim_k L_{a+nr}}{n^{\kappa}}
$$
exists for any fixed $a\in \NN$.
\end{Theorem}
The theorem says that 
$$
\lim_{n\rightarrow \infty}\frac{\dim_k L_{n}}{n^{\kappa}}
$$
exists if $n$ is constrained to lie in an arithmetic sequence $a+br$ with $r$ as above, and for some fixed $a$. The conclusions of the theorem are a little weaker than the conclusions of Theorem \ref{TheoremI1} for  varieties. In particular, the index $m(L)$ has little relevance on reduced but nonirreducible schemes (as shown by the example after Theorem \ref{Theorem8} and Example \ref{Example3}).

Now we turn to the case of  nonreduced proper schemes. We begin by returning to our discussion of  the relationship between growth rates and the Kodaira-Iitaka dimension of graded linear series, which is much more subtle on nonreduced schemes.

 Suppose that $X$ is a  proper scheme over a field $k$. Let $\mathcal N_X$ be the nilradical of $X$. Suppose that $L$ is a graded linear series on $X$. Then by Theorem \ref{Theorem8},
 there exists a positive constant $\gamma$ such that $\dim_kL_n<\gamma n^e$ where 
\begin{equation}\label{I4}
e=\max\{\kappa(L),\dim \mathcal N_X\}.
\end{equation}
This is the best bound possible. It is shown in Theorem \ref{TheoremN20} that if $X$ is a nonreduced projective $k$-scheme, then for any $s\in \NN\cup\{-\infty\}$ with $s\le \dim \mathcal N_X$, there exists a graded linear series $L$ on $X$ with $\kappa(L)=s$ and a constant $\alpha>0$ such that
$$
\alpha n^{\dim \mathcal N_X} < \dim_kL_n
$$
for all $n\gg 0$.

It follows from Theorem \ref{TheoremN1} that if $X$ is a proper $k$-scheme with $r=\dim \mathcal N_X\ge s$, then there exists a graded linear series $L$ of  Kodaira-Iitaka dimension $\kappa(L)=s$ such that the limit   
$$
\lim_{n\rightarrow \infty}\frac{\dim_k L_{n}}{n^{r}}
$$
does not exist, even when $n$ is constrained to lie in {\it any} arithmetic sequence. However, we prove in Theorem \ref{Theorem8}  that if 
$L$ is a graded linear series on a proper scheme $X$ over a field $k$, with $\kappa(L)>\dim \mathcal N_X$, then 
there exists a positive integer $r$ such that 
$$
\lim_{n\rightarrow \infty}\frac{\dim_k L_{a+nr}}{n^{\kappa(L)}}
$$
exists for any fixed $a\in \NN$.

This is the strongest statement on limits that is true. In fact, the existence of all such limits characterizes the dimension of the nilradical, at least on projective schemes.
We show this in the following theorem.

\begin{Theorem}(Theorem \ref{TheoremN2})
Suppose that $X$ is a $d$-dimensional projective scheme over a field $k$ with $d>0$. Let $\mathcal N_X$ be the nilradical of $X$. Let $\alpha\in\NN$. Then the following are equivalent:
\begin{enumerate}
\item[1)] For every graded linear series $L$ on $X$ with $\alpha\le\kappa(L)$, there exists a positive integer $r$ such that 
$$
\lim_{n\rightarrow\infty}\frac{\dim_kL_{a+nr}}{n^{\kappa(L)}}
$$
exists for every positive integer $a$.
\item[2)] For every graded linear series $L$ on $X$ with $\alpha\le \kappa(L)$, there exists an arithmetic sequence $a+nr$ (for fixed $r$ and $a$ depending on $L$) such that 
$$
\lim_{n\rightarrow\infty}\frac{\dim_kL_{a+nr}}{n^{\kappa(L)}}
$$
exists.
\item[3)] The nilradical $\mathcal N_X$ of $X$ satisfies $\dim \mathcal N_X<\alpha$.
\end{enumerate}
\end{Theorem}

If $X$ is a proper $k$-scheme of  dimension $d=0$ which is not irreducible, then the conclusions of Theorem \ref{TheoremN2} are true for $X$. This follows from Section \ref{SecZero}. However, 2) implies 3)  does not hold if $X$ is irreducible of dimension 0. In fact (Proposition \ref{TheoremN3}) if $X$ is irreducible of dimension 0,
and $L$ is a graded linear series on $X$ with $\kappa(L)=0$,  then there exists a positive integer $r$ such that the limit 
$\lim_{n\rightarrow \infty}\dim_kL_{a+nr}$ exists for every positive integer $a$.

\subsection{Volumes of line bundles}

The volume of a line bundle $\mathcal L$ on a $d$-dimensional proper variety $X$ is the limsup
\begin{equation}\label{in2}
\mbox{Vol}(\mathcal L)=\limsup_{n\rightarrow\infty}\frac{h^0(X,\mathcal L^n)}{n^d/d!}.
\end{equation}
There has  been much progress of our understanding of  the volume as a function on the big cone in $N^1(X)$ on a projective variety $X$
over an algebraically closed field (where (\ref{in2}) is actually a limit). Much of the  theory  is explained in \cite{La}, where extensive  references are given. Volume is continuous  on $N^1(X)$ but is not twice differentiable on all of $N^1(X)$ (as shown in an example of Ein Lazarsfeld, Musta\c{t}\u{a}, Nakamaye and Popa, \cite{ELMNP1}).   Boucksom, Favre and Jonsson \cite{BFJ} have shown that the volume is $\mathcal C^1$-differentiable on the big cone of $N^1(X)$ (when $X$ is a proper variety over an algebraically closed field of characteristic zero). This theorem is proven for a proper variety over an arbitrary field in \cite{C2}. The Fujita approximation type theorem Theorem \ref{FApp}, which is valid over an arbitrary field, is an ingredient of the proof.  Interpretation of the directional derivative in terms of intersection products and many  applications are given in \cite{BFJ}, \cite{ELMNP1}, \cite{LM} and \cite{C2}.

The starting point of the theory of volume on nonreduced schemes is to determine if the limsup defined in (\ref{in2}) exists as a limit. In Theorem \ref{Theorem14a}, it is shown that the volume of a line bundle always exists on a $d$-dimensional proper scheme $X$ over a field $k$ with $\dim \mathcal N_X<d$ (as explained earlier, this result was known on varieties over an algebraically closed field). We see from Theorem \ref{TheoremN1} and Example \ref{Example2}  that the limit does  not always exist for graded linear series $L$.

\section{notation and conventions}\label{SecNot} $m_R$ will denote the maximal ideal of a local ring $R$. $Q(R)$ will denote the quotient field of a domain $R$.
$\ell_R(N)$ will denote the length of an $R$-module $N$.  $\ZZ_+$ denotes the positive integers and $\NN$ the nonnegative integers. 
Suppose that $x\in \RR$. $\lceil x \rceil$ is the smallest integer $n$ such  that $x\le n$. $\lfloor x \rfloor$ is the largest integer $n$ such that $n\le x$. 

We recall some notation on multiplicity from Chapter VIII, Section 10 of \cite{ZS2}, Section V-2 \cite{Se} and Section 4.6 \cite{BH}.
Suppose that $(R,m_R)$ is a (Noetherian) local ring,  $N$ is a finitely generated $R$-module with $r=\dim N$ and $a$ is an ideal of definition of $R$. Then
$$
e_a(N)=\lim_{k\rightarrow \infty}\frac{\ell_R(N/a^kN)}{k^r/r!}.
$$
We write $e(a)=e_a(R)$.

If $s\ge r=\dim N$, then we define 
$$
e_s(a,N)=\left\{
\begin{array}{ll}
 e_a(N)&\mbox{ if }\dim N=s\\
 0&\mbox{ if } \dim N<s.
 \end{array}\right.
 $$

A local ring is analytically unramified if its completion is reduced. In particular, a reduced excellent local ring is 
analytically unramified.

 We will denote the maximal ideal of a local ring $R$ by $m_R$. If $\nu$ is a valuation of a field $K$, then we will write $V_{\nu}$ for the valuation ring of $\nu$, and $m_{\nu}$ for the maximal ideal of $V_{\nu}$. We will write $\Gamma_{\nu}$ for the value group of $\nu$. If $A$ and $B$ are local rings, we will say that $B$ dominates $A$ if $A\subset B$ and $m_B\cap A=m_A$.
 
 The dimension of an $R$-module $M$ is $\dim M=\dim R/\mbox{ann}(M)$.
 
 We  use the notation of Hartshorne \cite{H}. For instance, a variety is required to be integral.
If $\mathcal F$ is a coherent sheaf on a Noetherian scheme, then $\dim \mathcal F$ will denote the dimension of the support of $\mathcal F$, with  $\dim \mathcal F=-\infty$ if $\mathcal F=0$.

Suppose that $X$ is a scheme. The nilradical of $X$ is the ideal sheaf $\mathcal N_X$ on $X$ which is the kernel of the natural surjection
$\mathcal O_{X}\rightarrow \mathcal O_{X_{\rm red}}$ where $X_{\rm red}$ is the reduced scheme associated to $X$. $(\mathcal N_X)_\eta$ is the nilradical of the local ring $\mathcal O_{X,\eta}$ for all $\eta\in X$.

\section{Cones associated to semigroups}\label{SecCone}

In this section, we summarize some results of Okounkov \cite{Ok}, Lazarsfeld and Musta\c{t}\u{a} \cite{LM} and Kaveh and Khovanskii \cite{KK}. 

Suppose that $S$ is a subsemigroup of $\ZZ^{d}\times \NN$ which is not contained in $\ZZ^d\times\{0\}$. Let $L(S)$ be the subspace of $\RR^{d+1}$ which is generated by $S$, and let $M(S)=L(S)\cap(\RR^d\times\RR_{\ge 0})$. 

Let $\mbox{Con}(S)\subset L(S)$ be the closed convex cone which is the closure of  the set of all linear combinations $\sum \lambda_is_i$ with $s_i\in S$ and $\lambda_i\ge 0$.

$S$ is called {\it strongly nonnegative} (Section 1.4 \cite{KK}) if $\mbox{Cone}(S)$  intersects $\partial M(S)$ only at the origin (this is equivalent to being strongly admissible (Definition 1.9 \cite{KK}) since with our assumptions, $\mbox{Cone}(S)$ is contained in $\RR^d\times\RR_{\ge 0}$,  so the ridge of of $S$ must be contained in $\partial M(S)$). In particular, a subsemigroup of a strongly negative semigroup is itself strongly negative.

We now introduce some notation from \cite{KK}. Let 
\vskip .1truein

$S_k=S\cap (\RR^d\times\{k\})$.

$\Delta(S)=\mbox{Con}(S)\cap (\RR^{d}\times\{1\})$ (the Newton-Okounkov body of $S$).

$q(S)=\dim \partial M(S)$.

$G(S)$ be the subgroup of $\ZZ^{d+1}$ generated by $S$.

$m(S)=[\ZZ:\pi(G(S))]$
  where $\pi:\RR^{d+1}\rightarrow \RR$ be projection onto the last factor.

$\mbox{ind}(S)= [\partial M(S)_{\ZZ}:G(S)\cap \partial M(S)_{\ZZ}]$
where 

$\partial M(S)_{\ZZ}:=\partial M(S)\cap \ZZ^{d+1}= M(S)\cap (\ZZ^d\times\{0\})$.

${\rm vol}_{q(S)}(\Delta(S))$ is the integral volume of $\Delta(S)$. This volume is computed using the translation of the integral measure on $\partial M(S)$.
\vskip .2truein

$S$ is strongly negative if and only if $\Delta(S)$ is a compact set. If $S$ is strongly negative, then the dimension of $\Delta(S)$ is $q(S)$.
\vskip .1truein

\begin{Theorem}\label{ConeTheorem3}(Kaveh and Khovanskii) Suppose that $S$ is strongly nonnegative.  Then 
$$
\lim_{k\rightarrow \infty}\frac{\#S_{m(S)k}}{k^{q(S)}}=\frac{{\rm vol}_{q(S)}(\Delta(S))}{{\rm ind}(S)}.
$$
\end{Theorem}

This is proven in  Corollary 1.16 \cite{KK}. 

With our assumptions, we have that $S_n=\emptyset$ if $m(S)\not\,\mid n$ and  the limit is positive, since
${\rm vol}_{q(S)}(\Delta(S))>0$.

\begin{Theorem}\label{ConeTheorem1}(Okounkov,  Section 3 \cite{Ok}, Lazarsfeld and Musta\c{t}\u{a}, Proposition 2.1 \cite{LM}) Suppose that a subsemigroup $S$ of  $\ZZ^d\times\NN$  satisfies the following two conditions:
\begin{equation}\label{Cone2}
\begin{array}{l}
\mbox{There exist finitely many vectors $(v_i,1)$ spanning a semigroup $B\subset\NN^{d+1}$}\\
\mbox{such that $S\subset B$}
\end{array}
\end{equation}
and
\begin{equation}\label{Cone3}
G(S)=\ZZ^{d+1}.
\end{equation}
Then
$$
\lim_{n\rightarrow\infty} \frac{\# S_n}{n^d}={\rm vol}(\Delta(S)).
$$
\end{Theorem}

\begin{proof} 
$S$ is strongly nonnegative since $B$ is strongly nonnegative, so Theorem \ref{ConeTheorem3} holds.

$G(S)=\ZZ^{d+1}$ implies $L(S)=\RR^{d+1}$, so $M(S)=\RR^d\times\RR_{\ge 0}$, $\partial M(S)=\RR^d\times\{0\}$ and 
$q(S)=\dim \partial M(S)=d$. We thus have  $m(S)=1$ and $\mbox{ind}(S)=1$. 
\end{proof}

\begin{Theorem}\label{ConeTheorem4} Suppose that $S$ is  strongly nonnegative. Fix $\epsilon>0$. Then there is an integer $p=p_0(\epsilon)$ such that if $p\ge p_0$, then the limit
$$
\lim_{n\rightarrow\infty}\frac{\#(n*S_{pm(S)})}{n^{q(S)}p^{q(S)}}\ge \frac{{\rm vol}_{q(S)}\Delta(S)}{{\rm ind}(S)}-\epsilon
$$
exists, where
$$
n*S_{pm(S)}=\{x_1+\cdots+x_n\mid x_1,\ldots,x_n\in S_{pm(S)}\}.
$$
\end{Theorem}

\begin{proof} Let $m=m(S)$ and $q=q(S)$. Let $S^{[pm]}=\cup_{n=1}^{\infty} (n*S_{pm(S)})$ be the subsemigroup of $S$ generated by $S_{pm}$. For $p\gg 0$, we have that $L(S^{[pm]})=L(S)$ so $m(S^{[pm]})=pm$ and $q(S^{[pm]})=q$. 

Suppose that $v_1,\ldots,v_r$ generate $G(S)\cap \partial M(S)_{\ZZ}$. For $1\le i\le r$, there exist $a_i,b_i,n_i$ such that
$v_i=(a_i,n_im)-(b_i,n_im)$ with $(a_i,n_im), (b_i,n_im)\in S_{n_im}$. There exist $b>0$ and $c,c'$ such that $(c,mb) \in S$ and $(c',m(b+1))\in S$.
$bm$ divides $n_im+n_i(b-1)(b+1)m$ and
$$
v_i=[(a_i,n_im)+n_i(b-1)(c',(b+1)m)]-[(b_i,n_im)+n_i(b-1)(c',(b+1)m)],
$$
so we may assume that $b$ divides $n_i$ for all $i$. Thus $v_1,\ldots,v_r\in G(S^{[nm]})$ where $n=\max\{n_i\}$, and $v_1,\ldots,v_r\in G(S^{[pm]})$ whenever $p\ge (b-1)b+n$. Thus
\begin{equation}\label{eqnr70}
{\rm ind}(S^{[pm]})={\rm ind}(S)
\end{equation}
whenever $p\gg 0$.
We have that
\begin{equation}\label{eqnr71}
\lim_{p\rightarrow\infty} \frac{{\rm vol}_q(\Delta(S^{[pm]}))}{p^q}={\rm vol}_q(\Delta(S).
\end{equation}
By Theorem \ref{ConeTheorem3},
\begin{equation}\label{eqnr72}
\lim_{n\rightarrow \infty}\frac{\#(n*S_{pm})}{n^q}=\frac{{\rm vol}_q(\Delta(S^{[pm]}))}{{\rm ind}(S^{[pm]})}.
\end{equation}
The theorem now follows from (\ref{eqnr70}), (\ref{eqnr71}), (\ref{eqnr72}).
 \end{proof}
 
 We obtain the following result.

\begin{Theorem}\label{ConeTheorem2}(Proposition 3.1 \cite{LM}) Suppose that a subsemigroup $S$ of $\ZZ^d\times\NN$ satisfies  (\ref{Cone2}) and (\ref{Cone3}). Fix $\epsilon>0$. Then there is an integer $p_0= p_0(\epsilon)$
such that if $p\ge p_0$, then the limit
$$
\lim_{k\rightarrow \infty} \frac{\#(k*S_p)}{k^dp^d}\ge {\rm vol}(\Delta(S))-\epsilon
$$
exists.
\end{Theorem}

\section{Asymptotic theorems on lengths}\label{SecAsyRing}

\begin{Definition} A graded family of ideals $\{I_i\}$ in a ring $R$ is a family of ideals indexed by the natural numbers such that $I_0=R$ and $I_iI_j\subset I_{i+j}$
for all $i,j$.  If $R$ is a local ring and $I_i$ is $m_R$-primary for $i>0$, then we will say that $\{I_i\}$ is a graded family of $m_R$-primary ideals.
\end{Definition}

The following theorem is proven with the further assumptions that $R$ is equicharacteristic with perfect residue field in \cite{C1}.

\begin{Theorem}\label{Theorem1} Suppose that $R$ is an analytically irreducible local domain of dimension $d$ and  $\{I_i\}$ is a graded family of  $m_R$-primary ideals in $R$. 
Then
$$
\lim_{i\rightarrow\infty}\frac{\ell_R(R/I_i)}{i^d}
$$
exists.
\end{Theorem}

\begin{Corollary}\label{Cor33} Suppose that $R$ is an analytically irreducible local domain of dimension $d>0$ and  $\{I_i\}$ is a graded family of   ideals in $R$ such that there exists a positive number $c$ such that $m_R^c\subset I_1$.
Then
$$
\lim_{i\rightarrow\infty}\frac{\ell_R(R/I_i)}{i^d}
$$
exists.
\end{Corollary}

\begin{proof}
The assumption $m_R^c\subset I_1$ implies that either $I_n$ is $m_R$-primary for all positive $n$, or there exists
$n_0>1$ such that $I_{n_0}=R$. In the first case, the corollary follows from Theorem \ref{Theorem1}. In the second case, $m_R^{cn}\subset I_n$ for all $n\ge n_0$,  so $\ell_R(R/I_i)$ is actually bounded.
\end{proof}

We now give the proof of Theorem \ref{Theorem1}.

Since $I_1$ is $m_R$-primary, there exists   $c\in \ZZ_+$ such that 
\begin{equation}\label{eq8}
m_R^c\subset I_1.
\end{equation}

 Let $\hat R$ be the $m_R$-adic completion of $R$. Since the $I_n$ are $m_R$-primary, we have that 
$R/I_n\cong \hat R/I_n\hat R$ and $\ell_R(R/I_n)=\ell_{\hat R}(\hat R/I_n\hat R)$ for all $n$.   We may thus assume that $R$ is an excellent domain.
Let $\pi:X\rightarrow \mbox{spec}(R)$ be the normalization of the blow up of $m_R$. $X$ is of finite type over $R$ since $R$ is excellent.
Since $\pi^{-1}(m_R)$ has codimension 1 in $X$ and $X$ is normal, there exists a closed point $x\in X$ such that the local ring $\mathcal O_{X,x}$ is a regular local ring. Let $S$ be this local ring. $S$ is a regular local ring which  is essentially of finite type  and birational over $R$ ($R$ and $S$ have the same  quotient field).

Let $y_1,\ldots,y_d$ be a regular system of parameters in $S$. Let $\lambda_1,\ldots,\lambda_d$ be rationally independent real numbers, such that 
\begin{equation}\label{eq9}
\lambda_i\ge 1\mbox{ for all $i$}.
\end{equation}
 We define a valuation $\nu$ on
$Q(R)$ which dominates $S$ by prescribing 
$$
\nu(y_1^{a_1}\cdots y_d^{a_d})=a_1\lambda_1+\cdots+a_d\lambda_d
$$
for $a_1,\ldots,a_d\in \ZZ_+$, and $\nu(\gamma)=0$  if $\gamma\in S$ has nonzero residue in $S/m_S$.

Let $C$ be a coefficient set of $S$. Since $S$ is a regular local ring, for $r\in \ZZ_+$ and $f\in S$, there is a unique expression 
\begin{equation}\label{eqred11}
f=\sum s_{i_1,\ldots,i_d}y_1^{i_1}\cdots y_d^{i_d}+g_r
\end{equation}
with $g_r\in m_S^r$, $s_{i_1,\ldots,i_d}\in S$ and $i_1+\cdots+i_d<r$ for all $i_1,\ldots,i_d$ appearing in the sum. Take $r$ so large that 
$r> i_1\lambda_1+\cdots+i_d\lambda_d$ for some term with $s_{i_1,\ldots,i_d}\ne 0$. Then define
\begin{equation}\label{eqG61}
\nu(f)=\min\{i_1\lambda_1+\cdots+i_d\lambda_d\mid s_{i_1,\ldots,i_d}\ne 0\}.
\end{equation}
This definition is well defined, and we calculate that
$\nu(f+g)\ge \min\{\nu(f),\nu(g)\}$ and $\nu(fg)=\nu(f)+\nu(g)$ (by the uniqueness of the minimal value term in the expansion (\ref{eqred11})) for all $0\ne f,g\in S$. Thus $\nu$ is a valuation.
 Let $V_{\nu}$ be the valuation ring of $\nu$ (in $Q(R)$). The value group $\Gamma_{\nu}$ of $V_{\nu}$ is the (nondiscrete) ordered subgroup 
$\ZZ\lambda_1+\cdots+\ZZ\lambda_d$ of $\RR$. Since there is a unique monomial giving the minimum in (\ref{eqG61}), we have that the residue field of $V_{\nu}$ is
$S/m_S$.

Let $k=R/m_R$ and $k'=S/m_S=V_{\nu}/m_{\nu}$. Since $S$ is essentially of finite type over $R$, we have that $[k':k]<\infty$.

For $\lambda\in \RR$, define  ideals $K_{\lambda}$ and $K_{\lambda}^+$ in $V_{\nu}$ by
$$
K_{\lambda}=\{f\in Q(R)\mid \nu(f)\ge\lambda\}
$$
and
$$
K_{\lambda}^+=\{f\in Q(R)\mid \nu(f)>\lambda\}.
$$

We follow the usual convention that $\nu(0)=\infty$ is larger than any element of $\RR$. By Lemma 4.3 \cite{C1}, we have the following formula. The assumption that $R$ is analytically irreducible is necessary for the validity of the formula.

\begin{equation}\label{eqred50}
\mbox{There exists $\alpha\in \ZZ_+$ such that $K_{\alpha n}\cap R\subset m_R^n$ for all $n\in \NN$.}
\end{equation}

 Suppose that $I\subset R$ is an ideal and $\lambda\in \Gamma_{\nu}$ is nonnegative. Then we have inclusions of $k$-vector spaces
 $$
 I\cap K_{\lambda}/I\cap K_{\lambda}^+\subset K_{\lambda}/K_{\lambda}^+.
 $$
 Since $K_{\lambda}/K_{\lambda}^+$ is isomorphic to $k'$, we conclude that 
\begin{equation}\label{red1}
\dim_k I\cap K_{\lambda}/I\cap K_{\lambda}^+\le [k':k].
\end{equation}

 Let $\beta=\alpha c\in \ZZ_+$, where $c$ is the constant of (\ref{eq8}), and $\alpha$ is the constant of (\ref{eqred50}), so that for all $i\in \ZZ_+$,
\begin{equation}\label{eq13}
K_{\beta i}\cap R= K_{\alpha c i}\cap R\subset m_R^{ic}\subset I_i.
\end{equation}

For $t\ge  1$, define 
$$
\Gamma^{(t)}=\left\{
\begin{array}{l}
(n_1,\ldots,n_d,i)\in \NN^{d+1}\mid \dim_k I_i\cap K_{n_1\lambda_1+\cdots+n_d\lambda_d}/I_i\cap K_{n_1\lambda_1+\cdots+n_d\lambda_d}^+\ge t\\
\mbox{ and }n_1+\cdots+n_d\le \beta i
\end{array}
\right\},
$$
and
$$
\hat{\Gamma}^{(t)}=
\left\{
\begin{array}{l}(n_1,\ldots,n_d,i)\in \NN^{d+1}\mid \dim_k R\cap K_{n_1\lambda_1+\cdots+n_d\lambda_d}/R\cap K_{n_1\lambda_1+\cdots+n_d\lambda_d}^+\ge t\\
\mbox{ and }n_1+\cdots+n_d\le \beta i
\end{array}
\right\}.
$$

Let $\lambda=n_1\lambda_1+\cdots+n_d\lambda_d$ be such that $n_1+\cdots+n_d\le\beta i$. Then
\begin{equation}\label{eqred40}
\dim_kK_{\lambda}\cap I_i/K_{\lambda}^+\cap I_i=\#\{t|(n_1,\ldots,n_d,i)\in \Gamma^{(t)}\}.
\end{equation}

\begin{Lemma}\label{Lemmared1} Suppose that $t\ge 1$, $0\ne f\in I_i$, $0\ne g\in I_j$ and
$$
\dim_kI_i\cap K_{\nu(f)}/I_i\cap K_{\nu(f)}^+\ge t.
$$
Then
\begin{equation}\label{eqred10}
\dim_k I_{i+j}\cap K_{\nu(fg)}/I_{i+j}\cap K_{\nu(fg)}^+\ge t.
\end{equation}
In particular, when nonempty,
   $\Gamma^{(t)}$ and $\hat{\Gamma}^{(t)}$ are subsemigroups of the  semigroup $\ZZ^{d+1}$.
\end{Lemma}

\begin{proof} 
 There exist $f_1,\ldots,f_t\in I_i\cap K_{\nu(f)}$ such that their classes  are linearly independent over $k$ in $I_i\cap K_{\nu(f)}/ I_i\cap K_{\nu(f)}^+$.
 We will show that the classes of $gf_1,\ldots, gf_t$ in 
$$
I_{i+j}\cap K_{\nu(fg)}/I_{i+j}\cap K_{\nu(fg)}^+
$$
 are linearly independent over $k$.

Suppose that 
$a_1,\ldots,a_t \in k$  are such that the class of $a_1gf_1+\cdots +a_tgf_t$ in $I_{i+j}\cap K_{\nu(fg)}/I_{i+j}\cap K_{\nu(fg)}^+$ is zero. Then
$\nu(a_1gf_1+\cdots +a_tgf_t)>\nu(fg)$, whence $\nu(a_1f_1+\cdots+a_tf_t)>\nu(f)$,
so $a_1f_1+\cdots +a_tf_t\in I_i\cap K_{\nu(f)}^+$. Thus $a_1=\cdots=a_t=0$, since the classes of $f_1,\ldots,f_t$ are linearly independent over $k$ in $I_i\cap K_{\nu(f)}/ I_i\cap K_{\nu(f)}^+$. 
\end{proof}

From (\ref{eq13}), and since $n_1\lambda_1+\cdots+n_d\lambda_d<\beta i$
implies $n_1+\cdots+n_d< \beta i$   by (\ref{eq9}),
we have that 

\begin{equation}\label{eq12}
\begin{array}{lll}
\ell_R(R/I_i)&=&\ell_R(R/K_{\beta i}\cap R)-\ell_R(I_i/K_{\beta i}\cap I_i)\\
&=& \dim_k\left(\bigoplus_{0\le\lambda<\beta i}K_{\lambda}\cap R/K_{\lambda}^+\cap R\right)
-\dim_k\left(\bigoplus_{0\le \lambda<\beta i}K_{\lambda}\cap I_i/K_{\lambda}^+\cap I_i\right)\\
&=& \left(\sum_{t=1}^{[k':k]}\# \hat{\Gamma}_i^{(t)}\right)-\left(\sum_{t=1}^{[k':k]}\# \Gamma_i^{(t)}\right),
\end{array}
\end{equation}
where $\Gamma_i^{(t)}=\Gamma^{(t)}\cap (\NN^d\times\{i\})$ and $\hat{\Gamma}_i^{(t)}=\hat{\Gamma}^{(t)}\cap (\NN^d\times\{i\})$.
Since $R$ is Noetherian, there are only finitely many values of $\nu$ on $R$ which are $\le \beta i$,

For $0\ne f\in R$, define
$$
\phi(f)=(n_1,\ldots,n_d)\in \NN^d
$$
if $\nu(f)=n_1\lambda_1+\cdots+n_d\lambda_d$. We have that $\phi(fg)=\phi(f)+\phi(g)$.

\begin{Lemma}\label{Lemmared3} Suppose that $t\ge 1$ and $\Gamma^{(t)}\not\subset (0)$. Then  $\Gamma^{(t)}$ satisfies equations (\ref{Cone2}) and (\ref{Cone3}).
\end{Lemma}

\begin{proof} Let $\{e_i\}$ be the standard basis of $\ZZ^{d+1}$.
The semigroup
$$
B=\{(n_1,\ldots,n_d,i)\mid (n_1,\ldots,n_d)\in \NN^{d}\mbox{ and }n_1+\cdots+n_d\le\beta i\}
$$
is generated by $B\cap (\NN^d\times\{1\})$ and contains $\Gamma^{(t)}$, so (\ref{Cone2}) holds.

By assumption, there exists $r\ge 1$ and $0\ne h\in I_r$ such that $(\phi(h),r)\in \Gamma^{(t)}$.

There exists $0\ne u\in I_1$. 
Write $y_i=\frac{f_i}{g_i}$ with $f_i,g_i\in R$ for $1\le i\le d$.  Then $hf_i, hg_i\in I_{r}$. There exists $c'\in \ZZ_+$ such that $c'\ge c$ and $u, hf_i, hg_i\not\in m_R^{c'}$ for $1\le i\le d$. We may replace $c$ with $c'$ in (\ref{eq8}). Then
$(\phi(hf_i),r), (\phi(hg_i),r)\in \Gamma_r^{(t)}=\Gamma^{(t)}\cap(\NN^d\times\{r\})$ for $1\le i\le d$, by (\ref{eqred10}) (with $f_i,g_i\in I_0=R$)  since $hf_i$ and $hg_i$ all have values $n_1\lambda_1+\cdots+n_d\lambda_d<\beta r$, so that $n_1+\ldots+n_d<\beta r$. We have that  $\phi(y_i)=\phi(hf_i)-\phi(hg_i)=\phi(y_i)=e_i$ for $1\le i\le d$. Thus
$$
(e_i,0)=(\phi(hf_i),r)-(\phi(hg_i),r)\in G(\Gamma^{(t)})
$$
for $1\le i\le d$. 
 $(\phi(uh),r+1)\in \Gamma^{(t)}$ by (\ref{eqred10}) and since $\nu(u)\le \beta$,
so that $(\phi(u),1)\in G(\Gamma^{(t)})$, so $e_{d+1} \in G(\Gamma^{(t)})$. Thus $G(\Gamma^{(t)})=\ZZ^{d+1}$ and (\ref{Cone3}) holds.
\end{proof}

The  same argument proves the following lemma.
\begin{Lemma}\label{Lemmared4} Suppose that $t\ge 1$ and $\hat\Gamma^{(t)}\not\subset (0)$.  Then  $\hat{\Gamma}^{(t)}$ satisfies equations (\ref{Cone2}) and(\ref{Cone3}).
\end{Lemma}

By Theorem \ref{ConeTheorem1}, 
\begin{equation}\label{eq16}
\lim_{i\rightarrow\infty} \frac{\# \Gamma_i^{(t)}}{i^d}={\rm vol}(\Delta(\Gamma^{(t)}))
\end{equation}
and
\begin{equation}\label{eq17}
\lim_{i\rightarrow\infty} \frac{\# \hat{\Gamma}^{(t)}_i}{i^d}={\rm vol}(\Delta(\hat{\Gamma}^{(t)})).
\end{equation}
We obtain the conclusions of Theorem \ref{Theorem1} from equations (\ref{eq12}), (\ref{eq16}) and (\ref{eq17}).

\begin{Theorem}\label{Theorem2} Suppose that $R$ is a local ring of dimension $d$, and $\{I_i\}$ is a graded family of $m_R$-primary ideals in $R$. 
Let $N(\hat R)$ be the nilradical of the $m_R$-adic completion $\hat R$ of $R$, and suppose that $\dim N(\hat R)<d=\dim R$. Then
$$
\lim_{i\rightarrow \infty}\frac{\ell_R(R/I_i)}{i^d}
$$
exists.
\end{Theorem}

\begin{proof} Let $N=N(\hat R)$ and $A=\hat R/N$. We have a short exact sequence of $\hat R$-modules
$$
0\rightarrow N/(N\cap I_i\hat R)\rightarrow \hat R/I_i\hat R\rightarrow A/I_iA\rightarrow 0.
$$
There exists a number $c$ such that $m_R^c\subset I_1$. Hence $m_R^{ci}N\subset N\cap I_i\hat R$ for all $i$, so that
$$
\ell_{\hat R}(N/N\cap I_i\hat R)\le \ell_{\hat R}(N/m_{\hat R}^{ci}N)\le \alpha i^{\dim N}
$$
for some constant $\alpha$. 
Hence

$$
\lim_{i\rightarrow \infty}\frac{\ell_R(R/I_i)}{i^d}=\lim_{i\rightarrow \infty}\frac{\ell_{\hat R}(\hat R/I_i\hat R)}{i^d}
=\lim_{i\rightarrow \infty}\frac{\ell_A(A/I_iA)}{i^d}.
$$

 Let $p_1,\ldots, p_s$ be the minimal primes of $A$, and $A_j=A/p_j$ for $1\le j\le s$.
By Lemma \ref{Lemma5} below,
$$
\lim_{i\rightarrow \infty}\frac{\ell_{A}(A/I_iA)}{i^d}=\sum_{j=1}^s\lim_{i\rightarrow \infty}\frac{\ell_{A_j}(A_j/I_iA_j)}{i^d}
$$
which exists by Theorem \ref{Theorem1}
\end{proof}

\begin{Lemma}\label{Lemma5}(Lemma 5.1 \cite{C1}) Suppose that $R$ is a $d$-dimensional reduced local ring and $\{I_n\}$ is a graded family of $m_R$-primary ideals in $R$,  Let $\{p_1,\ldots, p_s\}$ be the minimal primes of $R$, $R_i=R/p_i$, and let $S$ be the ring
$S=\bigoplus_{i=1}^sR_i$. Then there exists $\alpha\in \ZZ_+$ such that for all $n\in \ZZ_+$,
$$
|\sum_{i=1}^s\ell_{R_i}(R_i/I_nR_i)-\ell_R(R/I_n)|\le \alpha n^{d-1}.
$$
\end{Lemma}

\section{A necessary and sufficient condition for limits to exist in a local ring}

Let $i_1=2$ and $r_1=\frac{i_1}{2}$. For $j\ge 1$, inductively define $i_{j+1}$ so that $i_{j+1}$ is even and $i_{j+1}>2^ji_j$.
Let $r_{j+1}=\frac{i_{j+1}}{2}$. For $n\in \ZZ_+$, define

\begin{equation}\label{eqsigma}
\sigma(n)=\left\{\begin{array}{ll}
1&\mbox{ if } n=1\\
\frac{i_j}{2} &\mbox{ if }i_j\le n<i_{j+1}
\end{array}\right.
\end{equation}

 \begin{Lemma}\label{Lemma1}
 Suppose that $a\in \NN$ and $r\in \ZZ_+$. Then given $m>0$ and $\epsilon>0$, there exists a positive integer $n=a+br$ with $b\in \NN$ such that $n>m$ and
 $$
 \left| \frac{\sigma(n)}{n}-\frac{1}{2}\right|<\epsilon
 $$
 \end{Lemma}
 
 \begin{proof} Choose $j$ sufficiently large that $i_j>m$, $i_j+r<i_{j+1}$  and
 \begin{equation}\label{eq1}
 \frac{i_j}{2(i_j+k)}> \frac{1}{2}-\epsilon
 \end{equation}
 for $0\le k<r$.
There exists $n=i_j+k$ with $0\le k<r$ in the arithmetic sequence $a+br$.
$$
\frac{\sigma(n)}{n}=\frac{i_j}{2n}=\frac{i_j}{2(i_j+k)}.
$$
By (\ref{eq1}),
$$
\frac{1}{2}\ge \frac{i_j}{2(i_j+k)}> \frac{1}{2}-\epsilon.
$$
\end{proof}

\begin{Lemma}\label{Lemma2}
 Suppose that $a\in \NN$ and $r\in \ZZ_+$. Then given $m>0$ and $\epsilon>0$, there exists a positive integer $n=a+br$ with $b\in \NN$ such that $n>m$ and
 $$
 \left| \frac{\sigma(n)}{n}\right|<\epsilon.
 $$
 \end{Lemma}

\begin{proof} Choose $j$ sufficiently large that $i_j>m+r$, $2^ji_j>r$ and
\begin{equation}\label{eq2}
\frac{i_j}{2(2^ji_j-k)}<\epsilon
\end{equation}
for $0<k\le r$. Let $n=i_{j+1}-k$ with $0<k\le r$ in the arithmetic sequence $a+br$.
$$
\frac{\sigma(n)}{n}=\frac{i_j}{2n}=\frac{i_j}{2(i_{j+1}-k)}.
$$
By (\ref{eq2}), 
$$
0<\frac{i_j}{2(i_{j+1}-k)}<\epsilon.
$$
\end{proof}

It follows from the previous two lemmas that 
the limit
\begin{equation}\label{eqnr1}
\lim_{n\rightarrow \infty} \frac{\sigma(n)}{n}
\end{equation}
does not exist, even when $n$ is constrained to lie in an  arithmetic sequence. 
The following example shows that limits might not exist on nonreduced local rings.

\begin{Example}\label{Example1} Let $k$ be a field, $d>0$  and $R$ be the nonreduced $d$-dimensional local ring $R=k[[x_1,\ldots,x_d,y]]/(y^2)$.
There exists a  graded family of $m_R$-primary ideals $\{I_n\}$ in $R$ such that the limit
$$
\lim_{n\rightarrow \infty} \frac{\ell_R(R/I_n)}{n^d}
$$
does not exist, even when $n$ is constrained to lie in an arithmetic sequence.
\end{Example}

\begin{proof}  Let $\overline x_1, \ldots,\overline x_d,\overline y$ be the classes of $x_1,\ldots,x_d,y$ in $R$.
Let $N_i$ be the set of monomials of degree $i$ in the variables $\overline x_1,\ldots,\overline x_d$.
Let $\sigma(n)$ be the function defined in (\ref{eqsigma}).
Define $M_R$-primary ideals $I_n$ in $R$ by
 $I_n=(N_n,\overline yN_{n-\sigma(n)})$ for $n\ge 1$ (and $I_0=R$). 

We first verify that $\{I_n\}$ is a graded family of ideals, by showing that $I_mI_n\subset I_{m+n}$ for all $m,n>0$.
This follows since
$$
I_mI_n=(N_{m+n},\overline yN_{(m+n)-\sigma(m)}, \overline yN_{(m+n)-\sigma(n)})
$$
and  $\sigma(j)\le \sigma(k)$ for $k\ge j$.

$R/I_n$ has a $k$-basis consisting of 
$$
\{N_i\mid i<n\}\mbox{ and }\{\overline yN_j\mid j<n-\sigma(n)\}.
$$
Thus 
$$
\ell_R(R/I_n)=\binom{n}{d}+\binom{n-\sigma(n)}{d}.
$$
does not exist, even when $n$ is constrained to lie in an arithmetic sequence, by (\ref{eqnr1}).
\end{proof}

Hailong Dao and  and Ilya Smirnov have communicated to me that they have   proven the following  theorem.

\begin{Theorem}\label{Theorem3}(Hailong Dao and Ilya Smirnov) Suppose that $R$ is a local ring of dimension $d>0$ with nilradical $N(R)$. Suppose that for any graded family $\{I_n\}$ of $m_R$-primary ideals, the limit 
$$
\lim_{n\rightarrow\infty}\frac{\ell_R(R/I_n)}{n^d}
$$
exists.
Then $\dim N(R)<d$.
\end{Theorem}

\begin{proof} Let $N=N(R)$. Suppose that $\dim N=d$. Let $p$ be a minimal  prime of $N$ such that $\dim R/p=d$. Then $N_p\ne 0$, so $p_p\ne 0$ in $R_p$. $p$ is an associated prime of $N$,  so there exists $0\ne x\in R$ such that $\mbox{ann}(x)=p$. $x\in p$, since otherwise 
$0=pxR_p=p_p$ which is impossible. In particular, $x^2=0$.

Let $f(n)=n-\sigma(n)$ be the function of (\ref{eqsigma}), and define $m_R$-primary ideals in $R$ by
$$
I_n=m_R^n+xm_R^{f(n)}.
$$
$\{I_n\}$ is a graded family of ideals in $R$ since 
$$
I_mI_n=(m_R^{m+n},xm_R^{(m+n)-\sigma(m)}, xm_R^{(m+n)-\sigma(n)})
$$
and  $\sigma(j)\le \sigma(k)$ for $k\ge j$. Let $\overline R=R/xR$. We have short exact sequences
\begin{equation}\label{eqnr2}
0\rightarrow xR/xR\cap I_n\rightarrow R/I_n\rightarrow \overline R/I_n\overline R\rightarrow 0.
\end{equation}
By Artin-Rees, there exists a number $k$ such that
$xR\cap m_R^n=m_R^{n-k}(xR\cap m_R^{n-k})$ for $n>k$. Thus $xR\cap m_R^n\subset xm_R^{f(n)}$ for $n\gg 0$ and $xR\cap I_n=xm_R^{f(n)}$ for $n\gg 0$.
We have that
$$
xR/xR\cap I_n\cong xR/xm_R^{f(n)}\cong R/({\rm ann}(x)+m_R^{f(n)})\cong R/p+m_R^{f(n)},
$$
so that  $\ell_R(xR/xR\cap I_n)=P_{R/p}(f(n))$  for $n\gg0$, where $P_{R/p}(n)$ is the Hilbert-Samuel polynomial of $R/p$. Hence
\begin{equation}\label{eqnr3}
\lim_{n\rightarrow \infty}\frac{\ell_R(xR/xR\cap I_n)}{n^d}=\frac{e(m_{R/p})}{d!}\lim_{n\rightarrow\infty}\left(\frac{f(n)}{n}\right)^d
\end{equation}
does not exist by (\ref{eqnr1}).
For $n\gg0$, 
$$
\ell_R(\overline R/I_n\overline R)=\ell_R(\overline R/m_{\overline R}^n)=P_{\overline R}(n)
$$
where $P_{\overline R}(n)$ is the Hilbert-Samuel polynomial of $\overline R$. Since $\dim \overline R\le d$, we have that
\begin{equation}\label{eqnr4}
\lim_{n\rightarrow \infty}\frac{\ell_R(\overline R/I_n\overline R)}{n^d}
\end{equation}
exists. Thus
$$
\lim_{n\rightarrow \infty}\frac{\ell_R(R/I_n)}{n^d}
$$
does not exist by (\ref{eqnr2}), (\ref{eqnr3}) and (\ref{eqnr4}).
\end{proof}

\begin{Theorem}\label{Theorem4} Suppose that $R$ is a local ring of dimension $d$, and  $N(\hat R)$ is the nilradical of the $m_R$-adic completion $\hat R$ of $R$.  Then   the limit 
$$
\lim_{n\rightarrow\infty}\frac{\ell_R(R/I_n)}{n^d}
$$
exists for any graded family $\{I_n\}$ of $m_R$-primary ideals, if and only if $\dim N(\hat R)<d$.
\end{Theorem}

\begin{proof} Sufficiency  follows from Theorem \ref{Theorem2}. Necessity  follows from  Theorem \ref{Theorem3} if $d>0$, since a family of 
$m_{\hat R}$-primary ideals in $\hat R$ naturally lifts to a graded family of $m_R$-primary ideals in $R$.

In the case when $d=0$ and $N(\hat R)\ne 0$, $R$ is an Artin local ring. Thus there exists some number $0<t$ such that $m_R^t\ne 0$ but $m_R^{t+1}=0$. With the notation before (\ref{eqsigma}), let
\begin{equation}\label{eqtau}
\tau(n)=\left\{\begin{array}{ll} 0&\mbox{ if $i_j\le n\le i_{j+1}$ and $j$ is even}\\
1&\mbox{ if $i_j\le n\le i_{j+1}$ and $j$ is odd.}\end{array}\right.
\end{equation}
Define a graded family of $m_R$-primary ideals $\{I_n\}$ in $R$ by $I_n=m_R^{t+\tau(n)}$. Then $\lim_{n\rightarrow \infty}\ell_R(R/I_n)$ does not exist.

\end{proof}

\begin{Corollary}\label{Cornr20}
Suppose that $R$ is an excellent local ring of dimension $d$, and  $N(R)$ is the nilradical of  $R$.  Then   the limit 
$$
\lim_{i\rightarrow\infty}\frac{\ell_R(R/I_n)}{n^d}
$$
exists for any graded family $\{I_i\}$ of $m_R$-primary ideals, if and only if $\dim N(R)<d$.
\end{Corollary}

\begin{proof} Let $N(\hat R)$ be the nilradical of $\hat R$. $\widehat{(R/N(R)}\cong \hat R/N(R)\hat R$ is reduced since $R/N(R)$ is (by Scholie IV.7.8.3 \cite{EGAIV}). Since $N(R)\hat R\subset N(\hat R)$, we have that $N(\hat R)=N(R)\hat R$. Thus $\mbox{gr}_{m_{\hat R}}(N(\hat R)) = \mbox{gr}_{m_R}(N(R)$,
so $\dim N(\hat R)=\dim N(R)$. Now the corollary follows from Theorem \ref{Theorem4}.
\end{proof}

\begin{Example} For any $d\ge 1$, there exists a local domain $R$ of dimension $d$ with a graded family of $m_R$-primary ideals $\{I_n\}$ such that the limit
$$
\lim_{n\rightarrow \infty}\frac{\ell_R(R/I_n)}{n^d}
$$
does not exist.
\end{Example}
\begin{proof}
The example of (E3.2) in \cite{N} is of a local domain $R$ such that the nilradical of $\hat R$ has dimension $d$. The example then follows from
Theorem \ref{Theorem3}  by lifting an appropriate graded family of $m_{\hat R}$-primary ideals to $R$.
\end{proof}

In Section 4 of \cite{CDK}, a series of examples of graded families of $m_R$-primary valuation ideals in a regular local ring $R$ of dimension two are given which have asymptotic growth of the rate $n^{\alpha}$, where $\alpha$ can be any rational number $1\le \alpha\le 2$. An example, also in a regular local ring of dimension two,  with growth  rate $n\mbox{log}_{10}(n)$ is given. Thus we generally do not have  a polynomial rate of growth.

\section{Applications to asymptotic multiplicities}\label{SecApp}

In this section, we apply Theorem \ref{Theorem1} and its method of proof, to generalize some of the applications in \cite{C1} to analytically  unramified local rings.

\begin{Theorem}\label{Theorem14} Suppose that $R$ is an analytically unramified  local ring of dimension $d>0$. 
Suppose that $\{I_i\}$ and $\{J_i\}$ are graded families of ideals in $R$. Further suppose that $I_i\subset J_i$ for all $i$ and there exists $c\in\ZZ_+$ such that
\begin{equation}\label{eq60}
m_R^{ci}\cap I_i= m_R^{ci}\cap J_i
\end{equation}
 for all $i$. 
 Assume that if $P$ is a minimal prime of $R$ then $I_1\subset P$ implies $I_i\subset P$ for all $i\ge 1$.
 Then the limit
$$
\lim_{i\rightarrow \infty} \frac{\ell_R(J_i/I_i)}{i^{d}}
$$
exists.
\end{Theorem}

Theorem \ref{Theorem14}  is proven for local rings $R$ which are regular, or normal excellent and equicharacteristic in \cite{C1}.

\begin{Remark} A reduced analytic local ring $R$ satisfies the hypotheses  of Theorem \ref{Theorem14}. 
In fact, an analytic local ring is excellent by Theorem 102 on page 291 \cite{Ma} and a reduced, excellent local ring is unramified by
(x) of Scholie 7.8.3 \cite{EGAIV}.
\end{Remark}

\begin{proof}  We may assume that $R$ is complete, by replacing $R$, $I_i$, $J_i$ by $\hat R$, $I_i\hat R$, $J_i\hat R$. 

First suppose that $R$ is analytically irreducible. Then either $I_i=J_i=0$ for all $i\ge 1$ or $I_1\ne 0$ (and $J_1\ne 0$). We may thus assume that $I_1\ne 0$. We will prove the theorem in this case. We will apply the method of Theorem \ref{Theorem1}.
 Construct the regular local ring $S$ by the argument of the proof of Theorem \ref{Theorem1}.

Let $\nu$ be the valuation of $Q(R)$ constructed from $S$ in the proof of Theorem \ref{Theorem1}, with associated valuation ideals $K_{\lambda}$ in the valuation ring $V_{\nu}$ of $\nu$. Let $k=R/m_R$ and $k'=S/m_S=V_{\nu}/m_{\nu}$.

By (\ref{eqred50}), there exists
$\alpha\in \ZZ_+$ such that 
$$
K_{\alpha n}\cap R\subset m_R^n
$$
for all $n\in \ZZ_+$. 
We have that 
$$
K_{\alpha cn}\cap I_n=K_{\alpha cn}\cap J_n
$$
for all $n$. Thus
\begin{equation}\label{eq31}
\ell_R(J_n/I_n)=\ell_R(J_n/K_{\alpha cn}\cap J_n)-\ell_R(I_n/K_{\alpha cn}\cap I_n)
\end{equation}
for all $n$. Let $\beta=\alpha c$. For $t\ge 1$, let
$$
\Gamma(J_*)^{(t)}=\{\begin{array}{l}
(n_1,\ldots,n_d,i)\mid \dim_k J_i\cap K_{n_1\lambda_1+\cdots+n_d\lambda_d}/J_i\cap K_{n_1\lambda_1+\cdots+n_d\lambda_d}^+\ge t\\
\mbox{ and }n_1+\cdots+n_d\le \beta i, 
\end{array}\}
$$
and
$$
\Gamma(I_*)^{(t)}=\{\begin{array}{l}
(n_1,\ldots,n_d,i)\mid \dim_k I_i\cap K_{n_1\lambda_1+\cdots+ n_d\lambda_d}/I_i\cap K_{n_1\lambda_1+\cdots+n_d\lambda_d}^+\ge t\\
\mbox{ and }n_1+\cdots+n_d\le \beta i
\end{array}\}.
$$
We have that 
\begin{equation}\label{eq32}
\ell_R(J_n/I_n)=(\sum_{t=1}^{[k':k]} \# \Gamma(J_*)^{(t)}_n)-(\sum_{t=1}^{[k':k]} \# \Gamma(I_*)^{(t)}_n)
\end{equation}
as explained in the proof of Theorem \ref{Theorem1}. As in the proof of Lemma \ref{Lemmared3} (this is were we need $I_1\ne 0$ and thus $J_1\ne 0$), we have that $\Gamma(J_*)^{(t)}$ and $\Gamma(I_*)^{(t)}$
  satisfy the conditions
(\ref{Cone2}) and (\ref{Cone3})of Theorem \ref{ConeTheorem1}. Thus
$$
\lim_{n\rightarrow \infty} \frac{\# \Gamma(J_*)^{(t)}_n}{n^d}={\rm vol}(\Delta(\Gamma(J_*))\mbox{ and }\lim_{n\rightarrow \infty} \frac{\# \Gamma(I_*)^{(t)}_n}{n^d}
={\rm vol}(\Delta(\Gamma(I_*)^{(t)})
$$
by Theorem \ref{ConeTheorem1}. 
The theorem, in the case when $R$ is analytically irreducible, now follows from (\ref{eq31}).

Now suppose that $R$ is only analytically unramified. We may continue to assume that $R$ is  complete. 
Let $P_1,\ldots,P_s$ be the minimal primes of $R$. Let $R_i=R/P_i$ for $1\le i\le s$. Let $T=\bigoplus_{i=1}^s R_i$ and $\phi:R\rightarrow T$ be the natural inclusion. By Artin-Rees, there exists a positive integer $\lambda$ such that
$$
\omega_n:=\phi^{-1}(m_R^nT)=R\cap m_R^nT\subset m_R^{n-\lambda}
$$
for all $n\ge \lambda$. 
Thus
$$
m_R^n\subset \omega_n\subset m_R^{n-\lambda}
$$
for all $n$.
We have that
$$
\omega_n=\phi^{-1}(m_R^nT)=\phi^{-1}(m_R^nR_1\bigoplus\cdots\bigoplus m_R^nR_s)
=(m_R^n+P_1)\cap \cdots \cap(m_R^n+P_s).
$$
Let $\beta=(\lambda+1)c$.
Now $\omega_{\beta n}\subset m_R^{c(\lambda+1)n-\lambda}\subset m_R^{cn}$ for all $n\ge 1$, so that 
$$
\omega_{\beta n}\cap I_n=\omega_{\beta n}\cap(m_R^{cn}\cap I_n)=\omega_{\beta n}\cap(m_R^{cn}\cap J_n)=\omega_{\beta n}\cap J_n
$$
for all $n\ge 1$, so 
\begin{equation}\label{eqNew1}
\ell_R(J_n/I_n)=\ell_R(J_n/\omega_{\beta n}\cap J_n)-\ell_R(I_n/\omega_{\beta n}\cap I_n)
\end{equation}
for all $n\ge 1$.

Define $L_0^j=R$ for $0\le j\le s$, and for $n>0$, define $L_n^0=J_n$ and
$$
L_n^j=(m_R^{\beta n}+P_1)\cap\cdots\cap (m_R^{\beta n}+P_j)\cap J_n
$$
for $1\le j\le s$. For fixed $j$, with $0\le j\le s$, $\{L_n^j\}$ is a graded family of ideals in $R$. For $n\ge 1$, we have isomorphisms
$$
L_n^j/L_n^{j+1}= L_n^j/(m_R^{\beta n}+P_{j+1})\cap L_n^j\cong L_n^jR_{j+1}/(L_n^jR_{j+1})\cap m_{R_{j+1}}^{\beta n}
$$
for $0\le j\le s-1$, and 
$$
L_n^s=\omega_{\beta n}\cap J_n.
$$
Thus
\begin{equation}\label{eqNew2}
\ell_R(J_n/\omega_{\beta n}\cap J_n)=\sum_{j=0}^{s-1}\ell_R(L_n^j/L_n^{j+1})
=\sum_{j=0}^{s-1}\ell_{R_{j+1}}\left(L_n^jR_{j+1}/(L_n^jR_{j+1})\cap m_{R_{j+1}}^{\beta n}\right).
\end{equation}

For some fixed $j$ with $0\le j\le s-1$, let $\overline R=R_{j+1}$, $\overline J_n=L_n^j\overline R$  and $\overline I_n =\overline J_n\cap m_{\overline R}^{\beta n}$.
$\{\overline I_n\}$ and $\{\overline J_n\}$ are graded families of ideals in  $\overline R$
and $m_{\overline R}^{\beta n}\cap\overline I_n=m_{\overline R}^{\beta n}\cap\overline J_n$ for all $n$. 

We have that $\overline I_1=0$ implies $\overline I_i=0$ for all $i\ge 1$ and $\overline J_1=0$ implies $\overline J_i=0$ for all $i\ge 1$ by our initial assumptions.
Since $\dim \overline R\le \dim R=d$ and $\overline R$ is  analytically irreducible, by the first part of the proof we have that
$$
\lim_{n\rightarrow \infty}\frac{\ell_{\overline R}(\overline J_n/\overline I_n)}{n^d}
$$
exists, and from (\ref{eqNew2}), we have that
$$
\lim_{n\rightarrow \infty}\frac{\ell_R(J_n/\omega_{\beta n}\cap J_n)}{n^d}
$$
exists. The same argument applied to the graded family of ideals $\{I_n\}$ in $R$ implies that 
$$
\lim_{n\rightarrow \infty}\frac{\ell_R(I_n/\omega_{\beta n}\cap I_n)}{n^d}
$$
exists. Finally, (\ref{eqNew1}) implies that the limit
$$
\lim_{n\rightarrow \infty}\frac{\ell_R(J_n/I_n)}{n^d}
$$
exists.

\end{proof}

If $R$ is a local ring and $I$ is an ideal in $R$ then the saturation of $I$ is 
$$
I^{\rm sat}=I:m_R^{\infty}=\cup_{k=1}^{\infty}I:m_R^k.
$$

\begin{Corollary}\label{Corollary5} Suppose that $R$ is an analytically unramified local ring of dimension $d>0$ and  $I$ is an ideal in $R$. Then the limit
$$
\lim_{i\rightarrow \infty} \frac{\ell_R((I^i)^{\rm sat}/I^i)}{i^{d}}
$$
exists.

\end{Corollary}

Since $(I^n)^{\rm sat}/I^n\cong H^0_{m_R}(R/I^n)$, the epsilon multiplicity of Ulrich and Validashti \cite{UV}
$$
\epsilon(I)=\limsup \frac{\ell_R(H^0_{m_R}(R/I^n))}{n^d/d!}
$$
exists as a limit, under the assumptions of Corollary \ref{Corollary5}.

Corollary \ref{Corollary5} is proven for more general families of modules when $R$ is a local domain which is essentially of finite type over a perfect field $k$ such that $R/m_R$ is algebraic over $k$ in \cite{C}. The corollary is proven with more restrictions on $R$ in Corollary 6.3 \cite{C1}. The limit in corollary \ref{Corollary5} can be irrational, as shown in \cite{CHST}.

\begin{proof} By Theorem 3.4 \cite{S}, there exists $c\in\ZZ_+$ such that each power $I^n$ of $I$ has an irredundant primary decomposition 
$$
I^n=q_1(n)\cap\cdots\cap q_s(n)
$$
where $q_1(n)$ is $m_R$-primary and $m_R^{nc}\subset q_1(n)$ for all $n$. Since $(I^n)^{\rm sat}=q_2(n)\cap \cdots\cap q_s(n)$,
we have that 
$$
I^n\cap m_R^{nc}=m_R^{nc}\cap q_2(n)\cap \cdots \cap q_s(n)=m_R^{nc}\cap (I^n)^{\rm sat}
$$
for all $n\in \ZZ_+$. Thus the corollary follows from Theorem \ref{Theorem14}, taking $I_i=I^i$ and $J_i=(I^i)^{\rm sat}$.

\end{proof}

A stronger version of the previous corollary is true. 
The following corollary proves a formula proposed by Herzog, Puthenpurakal and Verma in the introduction to \cite{HPV}.
The formula is proven with more restrictions on $R$ in Corollary 6.4 \cite{C1}.

Suppose that $R$ is a ring, and $I,J$ are ideals in $R$. Then the $n^{\rm th}$ symbolic power of $I$ with respect to $J$ is
$$
I_n(J)=I^n:J^{\infty}=\cup_{i=1}^{\infty}I^n:J^i.
$$

\begin{Corollary}\label{Cor5} Suppose that $R$ is an analytically unramified local ring of dimension $d$. 
 Suppose that $I$ and $J$ are ideals in $R$.  
 Let $s$ be the constant  limit dimension of $I_n(J)/I^n$ for $n\gg 0$. Suppose that $s<d$. Then
$$
\lim_{n\rightarrow \infty} \frac{e_{m_R}(I_n(J)/I^n)}{n^{d-s}}
$$
exists.
\end{Corollary}

\begin{proof} There exists a positive integer $n_0$ such that the  set of associated primes of $R/I^n$ stabilizes for
$n\ge n_0$ by  \cite{Br}. Let $\{p_1,\ldots, p_t\}$ be this set of associated primes.  We thus  have irredundant primary decompositions
for $n\ge n_0$,
\begin{equation}\label{eq**}
I^n=q_1(n)\cap \cdots \cap q_t(n),
\end{equation}
where $q_i(n)$ are $p_i$-primary.

We further have that 
\begin{equation}\label{eq*}
I^n:J^{\infty}=\cap_{J\not\subset p_i}q_i(n).
\end{equation}
Thus $\dim I_n(J)/I^n$ is constant for $n\ge n_0$. Let $s$ be this limit dimension. The set 
$$
A=\{p\in \cup_{n\ge n_0}{\rm Ass}(I_n(J)/I^n)\mid n\ge n_0\mbox{ and }\dim R/p=s\}
$$
is a finite set. Moreover, every such prime is in ${\rm Ass}(I_n(J)/I^n$ for all $n\ge n_0$. For $n\ge n_0$, we have 
by the additivity formula (V-2 \cite{Se} or Corollary 4.6.8, page 189 \cite{BH}), that
$$
e_{m_R}(I_n(J)/I^n)=\sum_{p}\ell_{R_p}((I_n(J)/I^n)_p)e(m_{R/p})
$$
where the sum is over the finite set of primes $p\in \mbox{Spec}(R)$ such that $\dim R/p=s$. This sum is thus over the finite set $A$.

Suppose that $p\in A$ and $n\ge n_0$. Then 
$$
I^n_p=\cap q_i(n)_p
$$
where the intersection is over the $q_i(n)$ such that $p_i\subset p$, and
$$
I_n(J)=\cap q_i(n)_p
$$
where the intersection is over the $q_i(n)$ such that $J\not\subset p_i$ and $p_i\subset p$. Thus  there exists an index $i_0$ such that $p_{i_0}=p$ and 
$$
I^n_p=q_{i_0}(n)_p\cap I_n(J)_p.
$$
By (\ref{eq**}), 
$$
(I^n_p)^{\rm sat}=I_n(J)_p
$$
for $n\ge n_0$. Since $R_p$ is analytically unramified (by \cite{R3} or Proposition 9.1.4 \cite{SH}) and $\dim R_p\le d-s$, by Corollary \ref{Corollary5}, the limit
$$
\lim_{n\rightarrow \infty} \frac{\ell_R((I_n(J)/I_n)_p)}{n^{d-s}}
$$
exists.
\end{proof}

\vskip .2truein
We now establish some Volume = Multiplicity formulas.

\begin{Theorem}\label{Theorem15} Suppose that $R$ is a $d$-dimensional analytically unramified local ring and  $\{I_i\}$ is a graded family of $m_R$-primary ideals in $R$. Then 
$$
\lim_{n\rightarrow \infty}\frac{\ell_R(R/I_n)}{n^d/d!}=\lim_{p\rightarrow \infty}\frac{e(I_p)}{p^d}
$$
exists.
Here $e(I_p)$ is the multiplicity
$$
e(I_p)=e_{I_p}(R)=\lim_{k\rightarrow \infty} \frac{\ell_R(R/I_p^k)}{k^d/d!}.
$$
\end{Theorem}

Theorem \ref{Theorem15} is proven for valuation ideals associated to an Abhyankar valuation in a regular local ring which is essentially of finite type over a field in  \cite{ELS}, for general families of $m_R$-primary ideals when $R$ is a regular local ring containing a field in \cite{Mus} and when $R$ is a local domain which is essentially of finite type over an algebraically closed field $k$ with $R/m_R=k$ in Theorem 3.8 \cite{LM}. 
It is proven when $R$ is regular or $R$ is analytically unramified with perfect residue field in Theorem 6.5 \cite{C1}.

\begin{proof} 
There exists $c\in\ZZ_+$ such that
$m_R^{c}\subset I_1$. 

We first prove the theorem when $R$ is analytically irreducible, and so satisfies the assumptions of Theorem \ref{Theorem1}. 
We may assume that $R$ is complete.
Let $\nu$ be the valuation of $Q(R)$ constructed from $S$ in the proof of Theorem \ref{Theorem1}, with associated valuation ideals $K_{\lambda}$ in the valuation ring $V_{\nu}$ of $\nu$. Let $k=R/m_R$ and $k'=S/m_S=V_{\nu}/m_{\nu}$.

Apply (\ref{eqred50}) to find
$\alpha\in \ZZ_+$ such that 
$$
K_{\alpha n}\cap R\subset m_R^n
$$
for all $n\in \NN$. 
We have that 
$$
K_{\alpha c n}\cap R \subset m_R^{cn}\subset I_n
$$
 for all $n$.

For $t\ge 1$, let
$$
\Gamma(I_*)^{(t)}=\{\begin{array}{l}
(n_1,\ldots,n_d,i)\mid \dim_k I_i\cap K_{n_1\lambda_1+\cdots+n_d\lambda_d}/I_i\cap K_{n_1\lambda_1+\cdots+n_d\lambda_d}^+\ge t\\
\mbox{ and }n_1+\cdots+n_d\le\alpha c i
\end{array} \},
$$
and
$$
\Gamma(R)^{(t)}=\{\begin{array}{l}
(n_1,\ldots,n_d,i)\mid \dim_k R\cap K_{n_1\lambda_1+\cdots+n_d\lambda_d}/R\cap K_{n_1\lambda-1+\cdots+n_d\lambda_d}^+\ge t\\
 \mbox{ and }n_1+\cdots+n_d\le\alpha c i
 \end{array}\}.
$$
As in the proofs of Lemmas \ref{Lemmared1} and \ref{Lemmared3},  $\Gamma(I_*)^{(t)}$ and $\Gamma(R)^{(t)}$ satisfy the conditions (\ref{Cone2}) and (\ref{Cone3})
of Theorem \ref{ConeTheorem1} when they are not contained in $\{0\}$.
For fixed $p\in \ZZ_+$ and $t\ge 1$, let
$$
\Gamma(I_*)(p)^{(t)}=\{\begin{array}{l}
(n_1,\ldots,n_d,kp)\mid \dim_k I_p^k\cap K_{n_1\lambda_1+\cdots+n_d\lambda_d}/I_p^k\cap K_{n_1\lambda_1+\cdots+n_d\lambda_d}^+\ge t\\
 \mbox{ and }n_1+\cdots+n_d\le\alpha c kp
 \end{array}\}.
$$
We have inclusions of semigroups
$$
k*\Gamma(I_*)_p^{(t)}\subset \Gamma(I_*)(p)^{(t)}_{kp}\subset \Gamma(I_*)^{(t)}_{kp}
$$
for all $p$, $t$ and $k$.

By Theorem \ref{ConeTheorem2}, given $\epsilon>0$, there exists $p_0$ such that $p\ge p_0$ implies
$$
{\rm vol}(\Delta(\Gamma(I_*)^{(t)})-\frac{\epsilon}{[k':k]}\le \lim_{k\rightarrow \infty}\frac{\#(k*\Gamma(I_*)^{(t)}_p)}{k^dp^d}.
$$
Thus
$$
{\rm vol}(\Delta(\Gamma(I_*)^{(t)})-\frac{\epsilon}{[k':k]}\le \lim_{k\rightarrow\infty}\frac{\#\Gamma(I_*)(p)^{(t)}_{kp}}{k^dp^d}
\le{\rm vol}(\Delta(\Gamma(I_*)^{(t)}).
$$
Again by Theorem \ref{ConeTheorem2}, we can choose $p_0$ sufficiently large that we also have that
$$
{\rm vol}(\Delta(\Gamma(R)^{(t)})-\frac{\epsilon}{[k':k]}\le\lim_{k\rightarrow \infty}\frac{\#\Gamma(R)^{(t)}_{kp}}{k^dp^d}\le{\rm vol}(\Delta(\Gamma(R)^{(t)})).
$$
Now 
$$
\ell_R(R/I_p^k)=(\sum_{t=1}^{[k':k]}\#\Gamma(R)^{(t)}_{kp})-(\sum_{t=1}^{[k':k]} \#\Gamma(I_*)(p)^{(t)}_{kp})
$$
and 
$$
\ell_R(R/I_n)=(\sum_{t=1}^{[k':k]} \#\Gamma(R)^{(t)}_n)-(\sum_{t=1}^{[k':k]} \#\Gamma(I_*)^{(t)}_n).
$$
By Theorem \ref{ConeTheorem1},
$$
\lim_{n\rightarrow \infty}\frac{\ell_R(R/I_n)}{n^d}=(\sum_{t=1}^{[k':k]}{\rm vol}(\Delta(\Gamma(R)^{(t)}))-(\sum_{t=1}^{[k':k]}{\rm vol}(\Delta(\Gamma(I_*)^{(t)}))).
$$
Thus
$$
\lim_{n\rightarrow\infty}\frac{\ell_R(R/I_n)}{n^d}-\epsilon\le \lim_{k\rightarrow\infty}\frac{\ell_R(R/I_p^k)}{k^dp^d}
=\frac{e(I_p)}{d!p^d}\le \lim_{n\rightarrow \infty}\frac{\ell_R(R/I_n)}{n^d}+\epsilon.
$$
Taking the limit as $p\rightarrow \infty$, we obtain the conclusions of the theorem.

Now assume that $R$ is analytically unramified. We may assume that $R$ is complete and reduced 
since 
$$
\ell_R(R/I_p^k)=\ell_{\hat R}(\hat R/I_p^k\hat R)\mbox{ and } e_d(I_p,R)=e_d(I_p\hat R,\hat R)
$$
for all $p,k$.

Suppose that the minimal primes of (the reduced ring) $R$ are $\{q_1,\ldots, q_s\}$. Let  $R_i=R/q_i$. $R_i$ are complete local domains.
We  have that
$$
\frac{e_d(I_p,R)}{p^d}=\sum_{i=1}^s\frac{e_d(I_pR_i,R_i)}{p^d}
$$
by the additivity  formula (page V-3 \cite{Se} or Corollary 4.6.8, page 189 \cite{BH}) or directly from Lemma \ref{Lemma5}. We also have  that
$$
\lim_{n\rightarrow\infty}\frac{\ell_R(R/I_n)}{n^d}=\sum_{i=1}^s\lim_{n\rightarrow\infty}\frac{\ell_R(R_i/I_nR_i)}{n^d}
$$
by Lemma \ref{Lemma5}. Since each $R_i$ is analytically irreducible, the limits 
$$
\lim_{n\rightarrow \infty}\frac{\ell_R(R_i/I_nR_i)}{n^d}=
\lim_{p\rightarrow \infty}\frac{e_d(I_pR_i,R_i)}{p^d}
$$
 exist by the earlier proof of this theorem for analytically irreducible local rings. The conclusions of the theorem now follow.

\end{proof}

Suppose that $R$ is a Noetherian ring, and $\{I_i\}$ is a graded  family of ideals in $R$.
Let
$$
s=s(I_*)=\limsup \dim R/ I_i.
$$
Let $i_0\in \ZZ_+$ be the smallest integer such that
\begin{equation}\label{eq40}
\mbox{$\dim R/I_i \le  s$ for $i\ge i_0$.}
\end{equation}
For $i\ge i_0$ and $p$ a prime ideal in $R$ such that $\dim R/p=s$, we have that $(I_i)_p=R_p$ or $(I_i)_p$ is $p_p$-primary.

$s$ is in general not a limit, as is shown by Example 6.6 \cite{C1}.

Let
$$
T=T(I_*)=\{p\in \mbox{spec}(R)\mid \dim R/p=s\mbox{ and there exist arbitrarily large $j$ such that $(I_j)_p\ne R_p$}\}.
$$

We recall some lemmas from \cite{C1}.

\begin{Lemma}\label{Lemma10}(Lemma 6.7 \cite{C1})  $T(I_*)$ is a finite set.
\end{Lemma}

\begin{Lemma}\label{Lemma11}(Lemma 6.8 \cite{C1}) There exist $c=c(I_*)\in \ZZ_+$ such that if $j\ge i_0$ and $p\in T(I_*)$, then 
$$
p^{jc}R_p\subset I_jR_p.
$$
\end{Lemma}

Let 
$$
A(I_*)=\{q\in T(I_*)\mid  \mbox{$I_nR_q$ is $q_q$-primary for  $n\ge i_0$}\}.
$$

\begin{Lemma}\label{Lemma50}(Lemma 6.9 \cite{C1})  Suppose that $q\in T(I_*)\setminus A(I_*)$. Then there exists $b\in \ZZ_+$ such that $q_q^b\subset (I_n)_q$ for all $n\ge i_0$.
\end{Lemma}

We obtain the following asymptotic additivity formula. It is proven in Theorem 6.10 \cite{C1}, with the additional assumption that $R$ is regular or analytically unramified of equicharacteristic zero.

\begin{Theorem}\label{Theorem13} Suppose that $R$ is a $d$-dimensional analytically unramified local ring 
and $\{I_i\}$ is a graded family of  ideals in $R$. Let
$s=s(I_*)=\limsup \dim R/I_i$ . Suppose that $s<d$. Then 
$$
\lim_{n\rightarrow \infty}\frac{e_{s}(m_R,R/I_n)}{n^{d-s}/(d-s)!}=\sum_{q}\left(\lim_{k\rightarrow \infty}\frac{e((I_k)_q)}{k^{d-s}}\right) e(m_{R/q})
$$
where the sum is over all prime ideals $q$ such that $\dim R/q=s$.
\end{Theorem}

\begin{proof} 
 Let $i_0$ be the (smallest) constant satisfying (\ref{eq40}). By the additivity formula (V-2 \cite{Se} or Corollary 4.6.8, page 189 \cite{BH}), for $i\ge i_0$,
$$
e_{s}(m_R,R/I_i)=\sum_p\ell_{R_p}(R_p/(I_i)_p)e_{m_R}(R/p)
$$
where the sum is over all prime ideals $p$ of $R$ with $\dim R/p=s$. By Lemma \ref{Lemma10}, for $i\ge i_0$,
the sum is actually over the finite set $T(I_*)$ of prime ideals of $R$.

For $p \in T(I_*)$, $R_p$ is a local ring of dimension $\le d-s$. Further, $R_p$  is analytically unramified
(by \cite{R3} or Prop 9.1.4 \cite{SH}). By Lemma \ref{Lemma11}, and by  Theorem \ref{Theorem2},  replacing $(I_i)_p$ with $p_p^{ic}$ if $i<i_0$, we have that
$$
\lim_{i\rightarrow \infty}\frac{\ell_{R_p}(R_p/(I_i)_p)}{i^{d-s}}
$$
exists. Further, this limit is zero if $p\in T(I_*)\setminus A(I_*)$ by Lemma \ref{Lemma50}, and since $s<d$. Finally, we have
$$
\lim_{i\rightarrow \infty}\frac{\ell_{R_q}(R_q/(I_i)_q)}{i^{d-s}/(d-s)!}=\lim_{k\rightarrow \infty}\frac{e_{(I_k)_q}(R_q)}{k^{d-s}}
$$
for $q\in A(I_*)$ by Theorem \ref{Theorem15}.

\end{proof}

\section{Kodaira-Iitaka dimension  on proper $k$-schemes}\label{SecProp}

Suppose that $X$ is a $d$-dimensional proper scheme over a field $k$, and $\mathcal L$ is a line bundle on $X$. 
Then under the natural inclusion of rings $k\subset \Gamma(X,\mathcal O_X)$, we have that the section ring
$$
\bigoplus_{n\ge 0}\Gamma(X,\mathcal L^n)
$$
is a graded $k$-algebra. Each $\Gamma(X,\mathcal L^n)$ is a finite dimensional $k$-vector space since $X$ is proper over $k$. In particular, $\Gamma(X,\mathcal O_X)$ is an Artin ring.     A graded $k$-subalgebra $L=\bigoplus_{n\ge 0}L_n$ of a section ring of a line bundle $\mathcal L$ on $X$ is called a {\it graded linear series} for $\mathcal L$. 

 We define the {\it Kodaira-Iitaka dimension} $\kappa=\kappa(L)$ of a graded linear series $L$  as follows.
Let
$$
\sigma(L)=\max \left\{m\mid 
\begin{array}{l}
 \mbox{there exists $y_1,\ldots,y_m\in L$ which are homogeneous of  positive}\\
\mbox{degree and are algebraically independent over $k$}
\end{array}\right\}.
$$
$\kappa(L)$ is then defined as
$$
\kappa(L)=\left\{\begin{array}{ll}
\sigma(L)-1 &\mbox{ if }\sigma(L)>0\\
-\infty&\mbox{ if }\sigma(L)=0
\end{array}\right.
$$

This definition is in agreement with the classical definition for line bundles on normal projective varieties (Definition in Section 10.1 \cite{I} or Chapter 2 \cite{La}).

\begin{Lemma}\label{LemmaKI} Suppose that $L$ is a graded linear series on a $d$-dimensional proper scheme $X$ over a field $k$. Then
\begin{enumerate}\item[1)]
\begin{equation}\label{eqKI1}
\kappa(L)\le d=\dim X.
\end{equation}
\item[2)] There exists a positive constant $\gamma$ such that 
\begin{equation}\label{eqKI4}
\dim_k L_n<\gamma n^d
\end{equation}
for all $n$.
\item[3)] Suppose that $\kappa(L)\ge 0$. Then there exists a positive constant $\alpha$ and a positive integer $e$ such that 
\begin{equation}\label{eqKI2}
\dim_kL_{en}>\alpha n^{\kappa(L)}
\end{equation}
for all positive integers $n$. 
\item[4)] Suppose that $X$ is reduced and $L$ is a graded linear series on $X$. Then $\kappa(L)=-\infty$ if and only if $L_n=0$ for all $n>0$.
\end{enumerate}
\end{Lemma}

We will show in Theorem \ref{Theorem8} that (\ref{eqKI4}) of Lemma \ref{LemmaKI} can be sharpened to the statement that there exists a positive constant $\gamma$ such that
\begin{equation}\label{eqN1}
\dim_kL_n<\gamma n^e
\end{equation}
where $e=\max\{\kappa(L),\dim \mathcal N_X\}$, where  $\mathcal N_X$ is the nilradical of $X$ (defined in the section on notations and conventions).
By Theorem \ref{TheoremN1}, (\ref{eqN1}) is the best bound possible.

To prove Lemma \ref{LemmaKI}, we need the following lemma.

\begin{Lemma}\label{LKI7} Suppose that $L$ is a graded linear series on  a projective  scheme $X$ over a field $k$. 
  Then 
  $$
  \sigma(L)={\rm Krull\,\, dimension}\,(L)
  $$
  and
\begin{equation}\label{eqKI7}
\kappa(L)=\left\{\begin{array}{ll}
{\rm Krull\,\, dimension}\,(L)-1&\mbox{ if }{\rm Krull\,\, dimension}\,(L)>0\\
-\infty&\mbox{ if }{\rm Krull\,\, dimension}\,(L)=0.
\end{array}\right.
\end{equation}
\end{Lemma} 

\begin{proof} We first prove the lemma with the assumption that $L$ is a finitely generated $k$-algebra.
In the case when $L_0=k$, the lemma  follows from  graded Noether normalization (Theorem 1.5.17 \cite{BH}). 
For a general graded linear series $L$, we always have that $k\subset L_0\subset \Gamma(X,\mathcal O_X)$, which is a finite dimensional $k$-vector space since $X$ is a projective $k$-scheme.  Let $m=\sigma(L)$ and $y_1,\ldots, y_m\in L$ be homogeneous elements of positive degree which are algebraically independent over $k$. Extend to homogeneous elements of positive degree $y_1,\ldots, y_n$ which generate $L$ as an $L_0$-algebra. 
Let $B=k[y_1,\ldots, y_n]$. We have that $\sigma(L)\le \sigma(B)\le \sigma(L)$ so $\sigma(B)=\sigma(L)$. By the first case ($L_0=k$) proven above, we have that 
$\sigma(B)=\mbox{Krull dimension}(B)$. 
 Since $L$ is finite over $B$, we have that $\mbox{Krull dimension}(L)=\mbox{Krull dimension}(B)$. Thus the lemma  holds when $L$ is a finitely generated $k$-algebra.
 
Now suppose that $L$ is an arbitrary graded linear series on $X$.  
Since $X$ is projective over $k$, we have an expression $X=\mbox{Proj}(A)$ where $A$ is the quotient of a standard graded polynomial ring 
$R=k[x_0,\ldots,x_n]$ by a homogeneous ideal $I$, which we can take to be saturated; that is, $(x_0,\ldots,x_n)$ is not an associated prime of $I$.
Let $\mathfrak p_1,\ldots,\mathfrak p_t$ be the associated primes of $I$. By graded prime avoidance (Lemma 1.5.10 \cite{BH}) there exists a form $F$ in $k[x_0,\ldots,x_n]$ of
some positive degree $c$ such that $F\not\in \cup_{i=1}^t\mathfrak p_i$. Then $F$ is a nonzero divisor on $A$, so that
$A\stackrel{F}{\rightarrow}A(c)$ is 1-1. Sheafifying, we have an injection
\begin{equation}\label{eqKI5}
0\rightarrow \mathcal O_X\rightarrow \mathcal O_X(c).
\end{equation}

Since $\mathcal O_X(c)$ is ample on $X$, there exists $f>0$ such that $\mathcal A := \mathcal L\otimes \mathcal O_X(cf)$ is ample.
From (\ref{eqKI5}) we then have
 a 1-1 $\mathcal O_X$-module homomorphism $\mathcal O_X\rightarrow \mathcal O_X(cf)$, and a 1-1 $\mathcal O_X$-module homomorphism
$\mathcal L\rightarrow \mathcal A$, which induces  inclusions of graded $k$-algebras
$$
L\subset \bigoplus_{n\ge 0}\Gamma(X,\mathcal L^n)\subset B:= \bigoplus_{n\ge 0}\Gamma(X,\mathcal A^n).
$$
There exists a positive integer $e$ such that $\mathcal A^e$ is very ample on $X$. Thus, by Theorem II.5.19 and Exercise II.9.9 \cite{H}, $B'=\bigoplus_{n\ge 0}B_{en}$ is finite over a coordinate ring $S$ of $X$
and  thus $B$ is a finitely generated $k$-algebra.
 
 Let $L^i$ be the $k$-subalgebra generated by $L_j$ for $j\le i$. 
 
 A $k$-algebra is {\it subfinite} if it is a subalgebra of a finitely generated $k$-algebra. We have  that 
 $$
 L^i\subset L\subset B
 $$
are subfinite $k$-algebras.  By Corollary 4.7 \cite{KT},
$$
{\rm Krull\,\, dimension}\,(L^i)\le {\rm Krull\,\, dimension}\,(L)\le {\rm Krull\,\, dimension}\,(B)=\dim(X)+1
$$
for all $i$.

Let 
$$
P_0\subset P_1\subset \cdots \subset P_r
$$
be a chain of distinct prime ideals in $L$ with $r= {\rm Krull\,\, dimension}\,(L)$. Since $\cup_{i=1}^{\infty} L^i=L$, there exists $n_0$ such that 
$$
P_0\cap L^i\subset P_1\cap L^i\subset \cdots \subset P_r\cap L^i
$$
is a chain of distinct prime ideals in $L^i$ for $i\ge n_0$, and so 
$$
{\rm Krull\,\, dimension}\,(L)={\rm Krull\,\, dimension}\,(L^i)
$$
for $\ge n_0$. For $i\gg 0$ we also have that $\sigma(L)=\sigma(L^i)$, so 
$$
\sigma(L)={\rm Krull\,\, dimension}\,(L).
$$
\end{proof}

We now give the proof of Lemma \ref{LemmaKI}.

Formula 2) follows from the following formula: Suppose that $\mathcal M$ is a coherent sheaf on $X$. Then there exists a positive constant $\gamma$ such that 
\begin{equation}\label{eqF}
\dim_k \Gamma(X,\mathcal M\otimes \mathcal L^n)<\gamma n^d
\end{equation}
 for all positive $n$.
 
 We first prove (\ref{eqF}) when $X$ is projective. Let $\mathcal O_X(1)$ be a very ample line bundle on $X$. By Proposition 7.4 \cite{H}, there
 exists a finite filtration of $\mathcal M$ by coherent sheaves $\mathcal M^i$ with quotients $\mathcal M^i/\mathcal M^{i-1}\cong \mathcal O_{Y_i}(n_i)$,
 where $Y_i$ are closed integral subschemes of $X$ and $n_i\in \ZZ$. There exists a number $c>0$ such that $\mathcal L\otimes \mathcal O_{Y_i}(c)$ is
 ample for all $i$. Let $\mathcal A=\mathcal O_X(n)\otimes \mathcal L$,
 where $n=c+\max\{|n_i|\}$. For all $i$ and positive $n$, we have 
 $$
 \dim_k\Gamma(X,(\mathcal M_i/\mathcal M_{i-1})\otimes \mathcal L^n)\le \dim_k\Gamma(Y_i,\mathcal O_{Y_i}\otimes \mathcal A^n). 
 $$
 This last is a polynomial in $n$ of degree equal to $\dim Y_i$ for large $n$ (by Proposition 8.8a \cite{I}).
 Thus we obtain the formula (\ref{eqF}) in the case that $X$ is projective. 
 
 Now suppose that $X$ is proper over $k$. We prove the formula by induction on $\dim \mathcal M$. If $\dim \mathcal M=0$, then $\dim_k\Gamma(X,\mathcal M)<\infty$, and $\mathcal M\otimes\mathcal L^n\cong\mathcal M$ for all $n$, so (\ref{eqF}) holds.
 Suppose that $\dim \mathcal M=e\,(\le d)$ and the formula is true for coherent $\mathcal O_X$-modules whose support has dimension $<e$. Let $\mathcal I$ be the sheaf of ideals on $X$ defined for $\eta\in X$ by
 $$
 \mathcal I_{\eta}=\{f\in \mathcal O_{X,\eta}\mid f\mathcal M_{\eta}=0\}.
 $$
 Let $Y=\mbox{Spec}(\mathcal O_X/\mathcal I)$, a closed subscheme of $X$. $\mathcal M$ is a coherent $\mathcal O_Y$-module, and
 $Y$ and $\mathcal M$ have the same support, so $\dim Y=e$. By Chow's Lemma, there exists a proper morphism
 $\phi:Y'\rightarrow Y$ such that $Y'$ is projective over $k$ and $\phi$ is an isomorphism over an open dense subset of $Y$. Ket $\mathcal K$ be the kernel of the natural morphism of $\mathcal O_Y$-modules
 $$
 \mathcal M\rightarrow \phi_*\phi^*\mathcal M.
 $$
 $\dim\mathcal K<e$ since $\phi$ is an isomorphism over a dense open subset of $Y$. Let $\mathcal L'=\phi^*(\mathcal L\otimes\mathcal O_Y)$. We have inequalities
 $$
 \dim_k\Gamma(X,\mathcal M\otimes\mathcal L^n)\le \dim_k\Gamma(X,\mathcal K\otimes\mathcal L^n)+\dim_k\Gamma(Y',\phi^*(\mathcal M)\otimes(\mathcal L')^n)
 $$
 for all $n\ge 0$, so we get the desired upper bound of (\ref{eqF}).

Now we establish 1). Let $\kappa:=\kappa(L)$. Then there exists an inclusion of a weighted polynomial ring $k[x_0,\ldots,x_{\kappa}]$ into $L$. 
Let $f$ be the least common multiple of the degrees of the $x_i$. Let $Z=\mbox{Proj}(k[x_0,\ldots,x_{\kappa}])$. $\mathcal O_Z(f)$ is an ample line bundle on the $\kappa$-dimensional weighted projective space $Z$. Thus there exists a polynomial $Q(n)$ of degree $\kappa$ such that
$$
\dim_kk[x_0,\ldots,x_n]_{nf}=\dim_k\Gamma(Z,\mathcal O_Z(nf))=Q(n)\mbox{ for all }n\gg 0.
$$
Thus there exists a positive constant $\alpha$ such that 
$$
\dim_kL_{nf}\ge \alpha n^{\kappa}\mbox{ for }n\gg 0,
$$
whence $\kappa(L)\le d$ by 2).

We will now establish formula 3).
Suppose that $\kappa(L)\ge 0$.
Let $L^i$ be the $k$-subalgebra of $L$ generated by $L_j$ for $j\le i$. For $i$ sufficiently large, we have that $\kappa(L^i)=\kappa(L)$. For such an $i$, 
since $L^i$ is a finitely generated $L_0$-algebra, we have that there exists a number $e$ such that the Veronese algebra $L^*$ defined by
$L^*_n=(L^i)_{en}$ is generated as a $L_0$-algebra in degree 1. Thus, since $L_0$ is an Artin ring,  and $L^*$ has Krull dimension $\kappa(L)+1$ by (\ref{eqKI7}), $L^*$ has a Hilbert polynomial $P(t)$ of degree $\kappa(L)$, satisfying  
$\ell_{L_0}(L^*_n)=P(n)$ for $n\gg 0$ (Corollary to Theorem 13.2 \cite{Ma2}), where $\ell_{L_0}$ denotes length of an $L_0$ module, and thus
$\dim_kL^*_n=(\dim_kL_0)P(n)$ for $n\gg 0$. Thus there exists a positive constant $\alpha$ such that $\dim_kL^*_n >\alpha n^{\kappa(L)}$ for all $n$, and so
$$
\dim_kL_{en}>\alpha n^{\kappa(L)}
$$
for all positive integers $n$, which is formula (\ref{eqKI2}).

Finally, we will establish the fourth statement of the lemma.
Suppose that $X$ is reduced and $0\ne L_n$ for some $n>0$. Consider the graded $k$-algebra homomorphism
$\phi:k[t]\rightarrow L$ defined by $\phi(t)=z$ where $k[t]$ is graded by giving $t$ the weight $n$. The kernel of $\phi$ is weighted homogeneous, so it is either 0 or 
$(t^s)$ for some $s>1$.  Thus if $\phi$ is not 1-1 then there exists $s>1$ such that $z^s=0$ in $L_{ns}$. We will show that this cannot happen.
Since $z$ is a nonzero global section of $\Gamma(X,\mathcal L^n)$, there exists $Q\in X$ such that the image of $z$ in $\mathcal L^n_Q$ is $\sigma f$ where $f\in \mathcal O_{X,Q}$ is nonzero and $\sigma$ is a local generator of $\mathcal L^n_Q$.
 The image of $z^s$ in $\mathcal L^{sn}_Q=\sigma^s\mathcal O_{X,Q}$ is $\sigma^sf^s$. We have that $f^s\ne 0$ since
$\mathcal O_{X,Q}$ is reduced. Thus $z^s\ne 0$. We thus have that $\phi$ is 1-1, so $\kappa(L)\ge 0$.

\section{Limits of graded linear series on proper varieties over a field}\label{SecLim}

Suppose that $L$ is a graded linear series on a proper variety $X$ over a field $k$. 
The {\it index} $m=m(L)$ of $L$ is defined as the index of groups
$$
m=[\ZZ:G]
$$
where $G$ is the subgroup of $\ZZ$ generated by $\{n\mid L_n\ne 0\}$.

The following theorem has been proven by  Okounkov  \cite{Ok} for section rings of ample line bundles,  Lazarsfeld and Musta\c{t}\u{a} \cite{LM} for section rings of big line bundles, and for graded linear series by Kaveh and Khovanskii \cite{KK}. All of these proofs require the assumption that {\it $k$ is  algebraically closed}.  We prove the result here for an arbitrary base field $k$.

\begin{Theorem}\label{Theorem5}  Suppose that $X$ is a $d$-dimensional proper variety over a field $k$, and $L$ is a graded linear series on $X$ with Kodaira-Iitaka dimension  $\kappa=\kappa(L)\ge 0$. Let $m=m(L)$ be the index of $L$.  Then  
$$
\lim_{n\rightarrow \infty}\frac{\dim_k L_{nm}}{n^{\kappa}}
$$
exists. 
\end{Theorem}

In particular, from the definition of the index, we have that the limit
$$
\lim_{n\rightarrow \infty}\frac{\dim_k L_{n}}{{n}^{\kappa}}
$$
exists, whenever $n$ is constrained to lie in an arithmetic sequence $a+bm$ ($m=m(L)$ and $a$ an arbitrary but fixed constant), as $\dim_kL_n=0$ if $m\not\,\mid n$.

An example of a big line bundle where the limit in Theorem \ref{Theorem5} is an irrational number is given in Example 4 of Section 7 \cite{CS}.

It follows that $\dim_kL_n=0$ if $m\not\,\mid n$, and if $\kappa(L)\ge 0$, then there exist positive constants $\alpha<\beta$ such that
\begin{equation}\label{eq61}
\alpha n^{\kappa(L)}<\dim_kL_{nm}<\beta n^{\kappa(L)}
\end{equation}
for all sufficiently large positive integers $n$

The following theorem is proven by Kaveh and Khovanskii \cite{KK} when $k$ is an algebraically closed field (Theorem 3.3 \cite{KK}). We prove the theorem for an arbitrary field. Theorem \ref{Theorem100} is a global analog of Theorem \ref{Theorem15}.

\begin{Theorem}\label{Theorem100} Suppose that $X$ is a $d$-dimensional proper variety over a field $k$, and $L$ is a graded linear series on $X$ with Kodaira-Iitaka dimension $\kappa=\kappa(L)\ge 0$. Let $m=m(L)$ be the index of $L$.  
Let $Y_{nm}$ be the projective subvariety of $\PP^{\dim_kL_{nm}}$ that is the closure of the image of the rational map 
$L_{nm}:X \dashrightarrow \PP_k^{\dim_kL_{nm}-1}$. Let $\deg(Y_{nm})$ be the degree of $Y_{nm}$ in $\PP_k^{\dim_kL_{nm}-1}$.
Then  $\dim Y_{nm}=\kappa$ for $n\gg 0$ and 
$$
\lim_{n\rightarrow \infty}\frac{\dim_k L_{nm}}{n^{\kappa}}=\lim_{n\rightarrow\infty}\frac{\deg(Y_{nm})}{\kappa!n^{\kappa}}.
$$
\end{Theorem}

Letting $t$ be an indeterminate, $\deg(Y_{nm})$ is the multiplicity of the  graded $k$-algebra $k[L_{nm}t]$ (with elements of $L_{nm}t$ having degree 1).

The proof by  Kaveh and Khovanskii  actually is valid for a variety $X$ over an arbitrary field $k$,
 with the additional assumption that there exists a valuation $\nu$  of the function field $k(X)$ of $X$ 
such that the value group $\Gamma_{\nu}$ of $\nu$ is isomorphic to $\ZZ^d$  and the residue field $V_{\nu}/m_{\nu}=k$. 
The existence of such a valuation is always true if $k$ is algebraically closed. It is however a rather special condition over non closed fields, as is shown by the following proposition.

\begin{Proposition}\label{Prop1} Suppose that $X$ is a $d$-dimensional projective variety over a field $k$. Then there exists a valuation $\nu$ of the function field $k(X)$ of $X$ 
such that the value group $\Gamma_{\nu}$ of $\nu$ is isomorphic to $\ZZ^d$  and the residue field $V_{\nu}/m_{\nu}=k$ if and only if 
there exists a birational morphism $X'\rightarrow X$ of projective varieties such that there exists a nonsingular (regular) $k$-rational point $Q'\in X'$.
\end{Proposition} 

\begin{proof}First suppose  there exists a valuation $\nu$ of the function field $k(X)$ of $X$ 
such that the value group $\Gamma_{\nu}$ of $\nu$ is isomorphic to $\ZZ^d$ as a group and with residue field $V_{\nu}/m_{\nu}=k$. Then $\nu$ is an ``Abhyankar valuation''; that is 
$$
\mbox{trdeg}_kk(X)=d=0+d= \mbox{trdeg}_kV_{\nu}/m_{\nu}+\mbox{rational rank }\Gamma_{\nu},
$$
with $k=V_{\nu}/m_{\nu}$, so there exists a local uniformization of $\nu$ by \cite{KKu}. Let $Q$ be the center of $\nu$ on $X$, so that $V_{\nu}$ dominates $\mathcal O_{X,Q}$. $\mathcal O_{X,Q}$ is a localization of a $k$-algebra $k[Z]$ where $Z\subset V_{\nu}$ is a finite set. By Theorem 1.1 \cite{KKu}, there exists a regular local ring $R$ which is essentially of finite type over $k$ with quotient field $k(X)$ such that $V_{\nu}$ dominates $R$ and $Z\subset R$. Since $k[Z]\subset R$ and $V_{\nu}$ dominates $\mathcal O_{X,Q}$, we have that $R$ dominates $\mathcal O_{X,Q}$. The residue field $R/m_R=k$ since $V_{\nu}$ dominates $R$.
There exists a projective $k$-variety $X''$ such that
$R$ is the local ring of a closed $k$-rational point $Q'$ on $X''$, and the birational map $X''\dashrightarrow X$ is a morphism in a neighborhood of $Q'$. Let $X'$ be the graph of the birational correspondence between $X''$ and $X$.
Since $X''\dashrightarrow X$ is a morphism in a neighborhood of $Q'$, the projection of $X'$ onto $X''$ is an isomorphism in a neighborhood of $Q'$. We can thus
identify $Q'$ with a nonsingular $k$-rational point of $X'$.

Now suppose that there exists a birational morphism $X'\rightarrow X$ of projective varieties such that there exists a nonsingular  $k$-rational point $Q'\in X'$. 

Choose a regular system of parameters $y_1,\ldots, y_d$ in $R=\mathcal O_{X',Q'}$. $R/m_R=k(Q')=k$, so $k$ is  a coefficient field of $R$.
 We have that $\hat R=k[[y_1,\ldots,y_d]]$. We define a valuation $\hat \nu$ dominating $\hat R$ by stipulating that
 \begin{equation}
\hat\nu(y_i)= e_i\mbox{ for $1\le i\le d$}
\end{equation}
where $\{e_i\}$ is the standard basis of the totally ordered  group $(\ZZ^d)_{\rm lex}$, and
 $\hat\nu(c)=0$ if $c$ is a nonzero element of $k$.
 
If $f\in \hat R$ and $f=\sum c_{i_1,\ldots,i_d}y_1^{i_1}\cdots y_d^{i_d}$ with $c_{i_1,\ldots,i_d}\in k$, then 
$$
\hat \nu(f)=\min\{\nu(y_1^{i_1}\cdots y_d^{i_d})\mid c_{i_1,\ldots,i_d}\ne0\}.
$$
We let $\nu$ be the valuation of the function field $k(X)$ which is obtained by restricting $\nu$. The value group of $\nu$ is $(\ZZ^d)_{\rm lex}$.

Suppose that $h$ is in $k(X)$ and $\nu(h)=0$. Write $h=\frac{f}{g}$ where $f,g\in R$ and $\nu(f)=\nu(g)$. Thus in $\hat R$,
we have expansions $f=\alpha y_1^{i_1}\cdots y_d^{i_d} +f'$, $g=\beta y_1^{i_1}\cdots y_d^{i_d} +g'$ where $\alpha,\beta$ are nonzero elements of $k$, $\nu(y_1^{i_1}\cdots y_d^{i_d})=\nu(f)=\nu(g)$
and $\nu(f')>\nu(f)$, $\nu(g')>\nu(g)$. Let  $\gamma=\frac{\alpha}{\beta}$ in $k$. Computing $f-\gamma g$ in $\hat R$, we obtain that $\nu(f-\gamma g)>\nu(f)$, and thus the residue of $\frac{f}{g}$ in $V_{\nu}/m_{\nu}$ is equal to the residue of $\gamma$, which is in $k$. By our construction $k\subset V_{\nu}$. Thus the residue field $V_{\nu}/m_{\nu}=k$.
\end{proof}

We now proceed to prove Theorems \ref{Theorem5} and \ref{Theorem100}.

By Chow's Lemma, there exists a proper birational morphism $\phi:X'\rightarrow X$ which is an isomorphism over a dense open set, such that $X'$ is projective over $k$. Since $X$ is integral, we have an inclusion $\Gamma(X,\mathcal L^n)\subset \Gamma(X,\phi^*\mathcal L^n)$ for all $n$.
Thus $L$ is a graded linear series for $\phi^*\mathcal L$, on the projective variety $X'$. In this way, we can assume that $X$ is in fact projective over $k$.

By \cite{Z}, $X$ has a closed regular point $Q$ (even though there may be no points which are smooth over $k$ if $k$ is not perfect).
Let  $R=\mathcal O_{X,Q}$. $R$ is a $d$-dimensional regular local ring. Let $k'=k(Q)=R/m_R$.

Choose a regular system of parameters $y_1,\ldots, y_d$ in $R$. By a similar  argument to that of the proof of Theorem \ref{Theorem1}, we may
define a valuation $\nu$ of the function field $k(X)$ of $X$ dominating $R$, by stipulating that
 \begin{equation}\label{eq20}
\nu(y_i)= e_i\mbox{ for $1\le i\le d$}
\end{equation}
where $\{e_i\}$ is the standard basis of the totally ordered  group $\Gamma_{\nu}=(\ZZ^d)_{\rm lex}$, and
 $\nu(c)=0$ if $c$ is a unit in $R$. As  in the proof of Theorem \ref{Theorem1}, we have that the residue field of the valuation ring $V_{\nu}$ of $\nu$ is  $V_{\nu}/m_{\nu}=k(Q)=k'$.

$L$ is a graded linear series for some line bundle $\mathcal L$ on $X$.  Since $X$ is integral, $\mathcal L$ is isomorphic to an invertible sheaf
$\mathcal O_X(D)$ for some Cartier divisor $D$ on $X$. 
We can assume that $Q$ is not contained in the support of $D$, after possibly replacing $D$ with a Cartier divisor linearly equivalent to $D$.
We have an induced graded $k$-algebra isomorphism of section rings 
$$
\bigoplus_{n\ge 0}\Gamma(X,\mathcal L^n)\rightarrow \bigoplus_{n\ge 0}\Gamma(X,\mathcal O_X(nD))
$$
which takes $L$ to a graded linear series for $\mathcal O_X(D)$. Thus we may assume that $\mathcal L=\mathcal O_X(D)$.
 For all $n$, the restriction map followed by inclusion into $V_{\nu}$,
\begin{equation}\label{eqR3}
\Gamma(X,\mathcal L^n)\rightarrow \mathcal L_Q=\mathcal O_{X,Q}\subset V_{\nu}
\end{equation}
 is a 1-1 $k$-vector space homomorphism since $X$ is integral, and we have an induced $k$-algebra homomorphism (sending $t\mapsto 1$). 
$$
 L\rightarrow \mathcal O_{X,Q}\subset V_{\nu}.
 $$

Given a nonnegative element $\gamma$ in the  value group $\Gamma_{\nu}=(\ZZ^d)_{\rm lex}$ of $\nu$, we have associated valuation ideals $I_{\gamma}$ and $I_{\gamma}^+$ in $V_{\nu}$ defined by 
$$
I_{\gamma}=\{f\in V_{\nu}\mid \nu(f)\ge \gamma\}
$$
and
$$
I_{\gamma}^+=\{f\in V_{\nu}\mid \nu(f)>\gamma\}.
$$
Since $V_{\nu}/m_{\nu}=k'$,  we have  that $I_{\lambda}/I_{\lambda}^+\cong k'$ for all nonnegative elements $\lambda\in \Gamma_{\nu}$, so
\begin{equation}\label{eqR1}
\dim_k(I_{\gamma}/I_{\gamma}^+)= [k':k]< \infty
\end{equation}
for all non negative $\gamma \in \Gamma_{\nu}$.
For $1\le t$, let
$$
S(L)_n^{(t)}=\{\gamma\in \Gamma_{\nu}\mid \dim_k L_n\cap I_{\gamma}/L_n\cap I_{\gamma}^+\ge t\}.
$$
Since every element of $L_n$ has non negative value (as $L_n\subset V_{\nu}$), we have by (\ref{eqR1}) and  (\ref{eqR3}) that 
\begin{equation}\label{eqR2}
\dim_kL_n=\sum_{t=1}^{[k':k]}\#(S(L)_n^{(t)})
\end{equation}
for all $n$.
For $1\le t$, let 
$$
S(L)^{(t)}=\{(\gamma,n)|\gamma \in S(L)_n^{(t)}\}.
$$
We have inclusions of semigroups
$S(L)^{(t')}\subset S(L)^{(t)}$ if $t<t'$.

\begin{Lemma}\label{Lemmaproj1} Suppose that $t\ge 1$, $0\ne f\in L_i$, $0\ne g\in L_j$ and
$$
\dim_kL_i\cap I_{\nu(f)}/L_i\cap I_{\nu(f)}^+\ge t.
$$
Then
\begin{equation}\label{eqproj10}
\dim_k L_{i+j}\cap I_{\nu(fg)}/L_{i+j}\cap I_{\nu(fg)}^+\ge t.
\end{equation}
In particular,
 the $S(L)^{(t)}$  are subsemigroups of the  semigroup $\ZZ^{d+1}$ whenever $S(L)^{(t)}\ne \emptyset$. We have that $m(S(L)^{(t)})=m(S(L)^{(1)})$ and $q(S(L)^{(t)})=q(S(L)^{(1)})$ for all $t$ such that $S(L)^{(t)}\not\subset  \{0\}$.
\end{Lemma}

\begin{proof}
The proof of (\ref{eqproj10}) and that $S(L)^{(t)}$ are sub semigroups is similar to that of Lemma \ref{Lemmared1}.

Suppose that $S(L^{(t)})\not\subset  \{0\}$. $S(L)^{(t)}\subset S(L)^{(1)}$, so 
\begin{equation}\label{eqnr60}
m(S(L)^{(1)})\mbox{ divides }m(S(L)^{(t)})
\end{equation}
and
\begin{equation}\label{eqnr61}
q(S(L)^{(t)})\le q(S(L)^{(1)}).
\end{equation}

For all $a\gg 0$, $S(L)^{(1)}_{am(S(L)^{(1)})}\ne \emptyset$. In particular, we can take  $a\equiv 1\,((\mbox{mod }mS(L)^{(t)}))$. There exists $b>0$ such that
$S(L)^{(t)}_{bm(S(L)^{(t)})}\ne \emptyset$. By (\ref{eqproj10}), we have that 
$$
S(L)^{(t)}_{am(S(L)^{(1)})+bm(S(L)^{(t)})}\ne \emptyset.
$$
Thus $m(S(L)^{(1)})\in \pi(G(S(L)^{(t)}))$, and by (\ref{eqnr60}), $m(S(L)^{(t)})=m(S(L)^{(1)})$.

Let $q=q(S(L)^{(1)})$. There exists $n_1>0$ and $(\gamma_1,n_1),\ldots, (\gamma_q,n_1)\in S(L)^{(1)}_{n_1}$ 
 such that if $\mathcal C_1$ is the cone generated by $(\gamma_1,n_1),\ldots, (\gamma_q,n_1)$ in $\RR^{d+1}$, then 
 $\dim \mathcal C_1\cap(\RR^d\times \{1\})=q$. There exists $(\tau, n_2)\in S(L^{(t)})$ with $n_2>0$. Thus 
 $$
 (\tau+\gamma_1, n_1+n_2),\ldots, (\tau+\gamma_q,n_1+n_2)\in S(L)^{(t)}_{n_1+n_2}
 $$
 by (\ref{eqproj10}). Let $\mathcal C_2$ be the cone generated by 
 $$
 (\tau+\gamma_1, n_1+n_2),\ldots, (\tau+\gamma_q,n_1+n_2)
 $$
 in $\RR^{d+1}$. Then $\dim \mathcal C_2\cap(\RR^d\times\{1\})=q$, and $q\le q(S(L)^{(t)})$. Thus, by (\ref{eqnr61}), $q(S(L)^{(t)})=q(S(L)^{(1)})$.

\end{proof}

We have that $m=m(L)$ is the common value of  $m(S(L)^{(t)})$. Let $q(L)$ be the common value of $q(S(L)^{(t)})$ for  $S(L)^{(t)}\not\subset \{0\}$.

There exists a very ample Cartier divisor $H$ on $X$ (at the beginning of the proof we reduced to $X$ being projective) such that $\mathcal O_X(D)\subset \mathcal O_X(H)$ and the point $Q$ of $X$ (from the beginning of the proof)
is  not contained in the support of $H$. Let $A_n=\Gamma(X,\mathcal O_X(nH))$ and $A$ be the section ring
$A=\bigoplus_{n\ge 0}A_n$. After possibly replacing $H$ with a sufficiently high multiple of $H$, we may assume that $A$ is generated in degree 1 as a $k''=\Gamma(X,\mathcal O_X)$-algebra. $[k'':k]<\infty$ since $X$ is projective. The $k$-algebra homomorphism $L\rightarrow V_{\nu}$ defined after (\ref{eqR3}) extends to a $k$-algebra homomorphism
$L\subset A\rightarrow V_{\nu}$. Let
$$
T_n=\{\gamma\in \Gamma_{\nu}\mid A_n\cap I_{\gamma}/A_n\cap I_{\gamma}^+\ne 0\},
$$
and $T=\{(\gamma,n)\mid \gamma\in T_n\}$. 
$T$ is a subsemigroup of $\ZZ^{d+1}$ by the argument of Lemma \ref{Lemmaproj1}, and we have inclusions of semigroups $S^{(t)}\subset T$ for all $t$.

By our construction, $A$ is naturally a graded subalgebra of the graded algebra $\mathcal O_{X,Q}[t]$. Since $H$ is ample on $X$, we have that $A_{(0)}=k(X)$, where $A_{(0)}$ is the set of elements of degree 0 in the localization of $A$ at the set of nonzero homogeneous elements of $A$. Thus for $1\le i\le d$, there exists $f_i,g_i\in A_{n_i}$, for some $n_i$, such that $\frac{f_i}{g_i}=y_i$. Thus
$$
(e_i,0)=(\nu(y_i),0)=(\nu(f_i),n_i)-(\nu(g_i),n_i)\in G(T).
$$
for $1\le i\le d$. Since $A_1\ne 0$, we then have that $(0,1)\in G(T)$. Thus $G(T)=\ZZ^{d+1}$, so $L(T)=\RR^{d+1}$, $\partial M(T)=\RR^d\times\{0\}$ and $q(T)=d$ (with the notation of  Section \ref{SecCone} on cones and semigroups).
For all $n\gg 0$ we have a bound
$$
|T_n|\le \dim_kA_n=[k'':k]\dim_{k''}A_n= [k'':k]P_A(n)
$$
where $P_A(n)$ is the Hilbert polynomial of the $k''$-algebra $A$, which has degree $\dim X=d=q(T)$. Thus, by Theorem 1.18 \cite{KK}, $T$ is a strongly nonnegative semigroup. Since the $S(L)^{(t)}$ are subsemigroups of $T$, they are also strongly nonnegative,
 so by Theorem \ref{ConeTheorem3} and (\ref{eqR2}), we have that
\begin{equation}\label{eqnr80}
\lim_{n\rightarrow \infty}\frac{\dim_kL_{nm}}{n^{q(L)}}=\sum_{t=1}^{[k':k]}\lim_{n\rightarrow \infty}\frac{\#(S(L)_{nm}^{(t)})}{n^{q(L)}}
=\sum_{t=1}^{[k':k]}\frac{{\rm vol}_{q(L)}(\Delta(S(L)^{(t)}))}{{\rm ind}(S(L)^{(t)})}
\end{equation}
exists.

Let $Y_{pm}$ be the varieties defined in the statement of Theorem \ref{Theorem100}. Let $d(pm)=\dim Y_{pm}$. The coordinate ring of $Y_{pm}$ is the $k$-subalgebra $L^{[pm]}:= k[L_{pm}]$ of $L$ (but with the grading giving elements of $L_{pm}$ degree 1). The Hilbert polynomial $P_{Y_{pm}}(n)$ of $Y_{pm}$ (Section I.7 \cite{H} or Theorem 4.1.3 \cite{BH}) has the properties that
\begin{equation}\label{eqred90}
P_{Y_{pm}}(n)=\frac{\deg(Y_{pm})}{d(pm)!}n^{d(pm)}+\mbox{lower order terms}
\end{equation}
and
\begin{equation}\label{eqred61}
\dim_kL^{[pm]}_{npm}=P_{Y_{pm}}(n)
\end{equation}
for $n\gg 0$. We have that
\begin{equation}\label{eqnr81}
\lim_{n\rightarrow \infty}\frac{\dim_k(L^{[pm]})_{npm}}{n^{d(pm)}}=\frac{\deg(Y_{pm})}{d(pm)!}.
\end{equation}

Suppose that $t$ is such that $1\le t\le [k':k]$ and $S(L)^{(t)}\not\subset \{ 0\}$. By Lemma \ref{Lemmaproj1}, for $p$ sufficiently large, we have that $m(S(L^{[pm]})^{(t)})=mp$. Let $\mathcal C$ be the closed cone generated by $S(L)^{(t)}_{pm}$ in $\RR^{d+1}$. We also have that 
$$
\dim (\mathcal C\cap (\RR^d\times\{1\}))=\dim(\Delta(S(L)^{(t)})=q(L)
$$
 for $p$ sufficiently large (the last equality is by Lemma \ref{Lemmaproj1}). Since
$S(L)^{(t)}_{pm}=S(L^{[pm]})^{(t)}_{pm}$, we have that 
$$
\dim (\mathcal C\cap (\RR^d\times\{1\}))\le \dim(\Delta(S(L^{[pm]})^{(t)})\le \dim (\Delta(L)^{(t)}).
$$
Thus 
\begin{equation}\label{eqred30}
q(S(L^{[pm]})^{(t)})=q(L)
\end{equation}
for all $p$ sufficiently large.

By the definition of Kodaira-Iitaka dimension, we also have that
\begin{equation}\label{eqred31} 
\kappa(L^{[pm]})=\kappa(L)
\end{equation}
for $p$ sufficiently large.

Now by graded Noether normalization (Section I.7 \cite{H} or Theorem 1.5.17\cite{BH}), the finitely generated  $k$-algebra $L^{[pm]}$ satisfies
\begin{equation}\label{eqred63}
d(pm)=\dim Y_{pm}=\mbox{Krull dimension}(L^{[pm]})-1=\kappa(L^{[pm]}).
\end{equation}

 We have that
\begin{equation}\label{eqred62}
\frac{1}{[k':k]}\dim_k L^{[pm]}_{npm}\le \#(S(L^{[pm]})^{(1)}_{npm})\le \dim_kL^{[pm]}_{npm}
\end{equation}
for all  $n$. $S(L^{[pm]})^{(1)}$ is strongly nonnegative since $S(L^{[pm]})^{(1)}\subset S(L)^{(1)}$ (or since $L^{[pm]}$ is a finitely generated $k$-algebra).
It follows from Theorem \ref{ConeTheorem3}, (\ref{eqred62}), (\ref{eqred61}), (\ref{eqred90}) and (\ref{eqred63})  that 
\begin{equation}\label{eqred91}
q(S(L^{[pm]})^{(1)})=d(pm)=\kappa(L^{[pm]}).
\end{equation}
From (\ref{eqred30}), (\ref{eqred91}) and (\ref{eqred31}), we have that 
\begin{equation}\label{eqnr85}
q(L)=\kappa(L)=\kappa.
\end{equation}

Theorem \ref{Theorem5} now follows from (\ref{eqnr80}) and (\ref{eqnr85}).
We now prove Theorem \ref{Theorem100}.  For all $p$, we have inequalities
$$
\sum_{t=1}^{[k':k]}\#(n*S(L)^{(t)}_{mp})\le \sum_{t=1}^{[k':k]}\#(S(L^{[mp]})^{(t)}_{nmp})\le \sum_{t=1}^{[k':k]}\#(S(L)_{nmp}^{(t)}).
$$
The second term in the inequality is $\dim_k(L^{[pm]})_{nmp}$ and the third term is $\dim_kL_{nmp}$.
Dividing by $n^{\kappa}p^{\kappa}$, and taking the limit as $n\rightarrow \infty$, we obtain from 
Theorem \ref{ConeTheorem4}, (\ref{eqnr85}) and (\ref{eqnr80}) for the first term and (\ref{eqnr81}), (\ref{eqred31}) and (\ref{eqred63}) for the second term,
 that
for given $\epsilon >0$, we can take $p$ sufficiently large that
$$
\lim_{n\rightarrow\infty}\frac{\dim_kL_{nm}}{n^{\kappa}}-\epsilon\le \frac{\deg(Y_{pm})}{\kappa!p^{\kappa}}\le 
\lim_{n\rightarrow\infty}\frac{\dim_kL_{nm}}{n^{\kappa}}.
$$
Taking the limit as $p$ goes to infinity then proves Theorem \ref{Theorem100}.

\vskip .2truein

\begin{Theorem}\label{FApp}(Fujita Approximation)
Suppose that $D$ is a big Cartier divisor on a complete variety $X$  of dimension $d$ over a field $k$, and $\epsilon>0$ is given. Then there exists a projective variety $Y$ with a birational morphism $f:Y\rightarrow X$, a nef and big $\QQ$-divisor $N$ on $Y$, and an effective $\QQ$-divisor $E$ on $Y$ such that there exists $n\in \ZZ_{>0}$ so that $nD$, $nN$ and $nE$ are Cartier divisors with
 $f^*(nD)\sim nN+nE$, where $\sim$ denotes linear equivalence, and
 $$
 {\rm vol}_Y(N)\ge {\rm vol}_X(D)-\epsilon.
 $$
 \end{Theorem}
 
 \begin{proof} By taking a Chow cover by a birational morphism, which is an isomorphism in codimension one, we may assume that $X$ is projective over $k$. This theorem was proven over an algebraically closed field of characteristic zero by Fujita \cite{Fuj} (c.f. Theorem 10.35 \cite{La}). It is proven in Theorem 3.4 and Remark 3.4 \cite{LM} over an arbitrary algebraically closed field (using Okounkov bodies) and by Takagi \cite{T} using de Jong's alterations \cite{DJ}. 
 
We give a proof for an arbitrary field.  The conclusions of Theorem 3.3 \cite{LM} over an arbitrary field follow from Theorem \ref{Theorem100} and formula (\ref{eqnr81}), taking  the $L_n$ of Theorem \ref{Theorem100} to be the $H^0(X,\mathcal O_X(nD))$ of Theorem 3.3 \cite{LM}. $m=1$ in Theorem \ref{Theorem100} since $D$ is big.  Then the $V_{k,p}$ of Theorem 3.3 \cite{LM} are the $L_{kp}^{[p]}$ of the proof of Theorem \ref{Theorem100}.

The proof of Remark 3.4 \cite{LM} is valid over an arbitrary field, using the strengthened form of Theorem 3.3 \cite{LM} given above, from which the
approximation theorem follows.
 
\end{proof}

\section{Limits on reduced proper schemes over a field}\label{SecRed}

Suppose that $X$ is a proper scheme over a field $k$ and $L$ is a graded linear series for a line bundle $\mathcal L$ on $X$.
Suppose that $Y$ is a closed subscheme of $X$. Set $\mathcal L|Y=\mathcal L\otimes_{\mathcal O_X}\mathcal O_Y$. Taking global sections of the natural surjections
$$
\mathcal L^n\stackrel{\phi_n}{\rightarrow} (\mathcal L|Y)^n\rightarrow 0,
$$
for $n\ge 1$ we have  induced short exact sequences of $k$-vector spaces
\begin{equation}\label{eq54}
0\rightarrow K(L,Y)_n\rightarrow L_n\rightarrow (L|Y)_n\rightarrow 0,
\end{equation}
where 
$$
(L|Y)_n:=\phi_n(L_n)\subset \Gamma(Y,({\mathcal L}|Y)^n)
$$
 and $K(L ,Y)_n$ is the kernel of $\phi_n|L_n$. Defining $K(L,Y)_0=k$ and $(L|Y)_0=\phi_0(L_0)$, we have that
$L|Y=\bigoplus_{n\ge 0}(L|Y)_n$ is a graded linear series for $\mathcal L|Y$ and $K(L,Y)=\bigoplus_{n\ge 0}K(L,Y)_n$ is a graded linear series for $\mathcal L$.

\begin{Lemma}\label{Lemma50a}  Suppose that $X$ is a reduced proper scheme over a field $k$ and $X_1,\ldots,X_s$ are the irreducible components of $X$. Suppose that $L$ is a graded linear series on $X$. Then 
$$
\kappa(L)=\max\{\kappa(L|X_i)\mid 1\le i\le s\}.
$$
\end{Lemma}

\begin{proof} $L$ is a graded linear series  for a line bundle $\mathcal L$ on $X$.
Let $X_1,\ldots,X_s$ be the irreducible components of $X$. Since $X$ is reduced, 
we have a natural inclusion
$$
0\rightarrow \mathcal O_X\rightarrow \bigoplus_{i=1}^s \mathcal O_{X_i}.
$$
There is a  natural inclusion of $k$-algebras
$$
\bigoplus_{n\ge 0}\Gamma(X,\mathcal L^n)\rightarrow  \bigoplus_{i=1}^s\left(\bigoplus_{n\ge 0}\Gamma(X_i,\mathcal L^n\otimes_{\mathcal O_X}\mathcal O_{X_i})\right),
$$ 
which induces an inclusion of $k$-algebras
\begin{equation}\label{eq42}
L\rightarrow \bigoplus_{i=1}^sL|X_i.
\end{equation}
Suppose that $i$ is such that $1\le i\le s$. Set $t=\kappa(L|X_i)$. Then by the definition of Kodaira-Iitaka dimension, there exists a graded inclusion of $k$-algebras $\phi:k[z_1,\ldots,z_t]\rightarrow L|X_i$ where
$k[z_1,\ldots,z_t]$ is a graded polynomial ring. Since the projection $L\rightarrow L|X_i$ is a surjection, we have a lift of $\phi$ to a graded $k$-algebra homomorphism into $L$, which is 1-1, so that $\kappa(L)\ge t$. Thus
$$
\kappa(L)\ge \max\{\kappa(L|X_i)\mid 1\le i\le s\}.
$$
Let $\kappa=\kappa(L)$. Then there exists a 1-1 $k$-algebra homomorphism $\phi:k[z_1,\ldots,z_q]\rightarrow L$ where
$k[z_1,\ldots,z_q]$ is a positively graded polynomial ring. Let $\phi_i:k[z_1,\ldots,z_q]\rightarrow L|X_i$ be the induced homomorphisms, for $1\le i\le s$.
Let $\mathfrak p_i$ be the kernel of $\phi_i$. Since (\ref{eq42}) is 1-1, we have that $\mathfrak p_1\cap \cdots \cap \mathfrak p_s=(0)$.
Since $k[z_1,\ldots,z_q]$ is a domain, this implies that some $\mathfrak p_i=(0)$. Thus $\phi_i$ is 1-1 and we have that $\kappa(L|X_i)\ge \kappa(L)$.

\end{proof}

\begin{Theorem}\label{Theorem18} Suppose that $X$ is a reduced proper scheme over a  field $k$. 
Let $L$ be a graded linear series on $X$ with Kodaira-Iitaka dimension  $\kappa=\kappa(L)\ge 0$.  
 Then there exists a positive integer $r$ such that 
$$
\lim_{n\rightarrow \infty}\frac{\dim_k L_{a+nr}}{n^{\kappa}}
$$
exists for any fixed $a\in \NN$.
\end{Theorem}
The theorem says that 
$$
\lim_{n\rightarrow \infty}\frac{\dim_k L_{n}}{n^{\kappa}}
$$
exists if $n$ is constrained to lie in an arithmetic sequence $a+br$ with $r$ as above, and for some fixed $a$. The conclusions of the theorem are a little weaker than the conclusions of Theorem \ref{Theorem5} for  varieties. In particular, the index $m(L)$ has little relevance on reduced but not irreducible schemes (as shown by the example after Theorem \ref{Theorem8} and Example \ref{Example3}.

\begin{proof}  Let $X_1,\ldots,X_s$ be the irreducible components of $X$. 
 Define graded linear series $M^i$ on $X$
 by $M^0=L$, $M^i=K(M^{i-1},X_i)$ for $1\le i\le s$. 
 By (\ref{eq54}), for $n\ge 1$, we have exact sequences of $k$-vector spaces 
 $$
 0\rightarrow (M^{j+1})_n=K(M^j,X_{j+1})_n\rightarrow M_n^j\rightarrow (M^j|X_{j+1})_n\rightarrow 0
 $$
 for $0\le j\le s-1$, and thus
 $$
 M_n^j={\rm Kernel}(L_n\rightarrow \bigoplus_{i=1}^j(L|X_i)_n)
 $$
 for $1\le j\le s$. The natural map $L\rightarrow \bigoplus_{i=1}^sL|X_i$ is an injection of $k$-algebras since 
 $X$ is reduced. Thus $M_n^s=(0)$, and
 \begin{equation}\label{eq71}
 \dim_kL_n=\sum_{i=1}^{s}\dim_k (M^{i-1}|X_i)_n
 \end{equation}
 for all $n$.
  Let $r=\mbox{LCM}\{m(M^{i-1}|X_i)\mid  \kappa(M^{i-1}|X_i)=\kappa(L)\}$. 
 The theorem now follows from Theorem \ref{Theorem5} applied to each of the $X_i$ with $\kappa(M^{i-1}|X_i)=\kappa(L)$ (we can start with an $X_1$ with $\kappa(L|X_1)=\kappa(L)$).
\end{proof}

\begin{Corollary}\label{Cor72}
Suppose that $X$ is a reduced projective scheme over a perfect field $k$. 
Let $L$ be a graded linear series on $X$ with $\kappa(L)\ge 0$. Then there exists a positive constant $\beta$ such that 
\begin{equation}\label{eq80}
\dim_kL_n<\beta n^{\kappa(L)}
\end{equation}
for all $n$. Further, there exists a positive constant $\alpha$ and a positive integer $m$ such that
\begin{equation}\label{eq81}
\alpha n^{\kappa(L)}<\dim_kL_{mn}
\end{equation}
for all positive integers $n$.
\end{Corollary}

\begin{proof} Equation (\ref{eq80}) follows from (\ref{eq71}), since $\dim_k(M^{i-1}|X_i)\le \dim_k(L|X_i)$ for all $i$, and since (\ref{eq61}) holds on a variety. Equation (\ref{eq81}) is immediate from (\ref{eqKI2}).
\end{proof}

The following lemma is required for the construction of the next example. It follows from Theorem V.2.17 \cite{H} when $r=1$.  The lemma uses the notation of \cite{H}.

\begin{Lemma}\label{Lemma52} Let $k$ be an algebraically closed field, and write $\PP^1=\PP^1_k$.
Suppose that $r\ge 0$. Let $X=\PP(\mathcal O_{\PP^1}(-1)\bigoplus \mathcal O_{\PP^1}^r)$ with natural projection $\pi:X\rightarrow \PP^1$. Then the complete linear
system $|\Gamma(X,\mathcal O_X(1)\otimes \pi^*\mathcal O_{\PP^1}(1))|$ is base point free, and the only curve contracted by the induced morphism of $X$ is the curve $C$ which is the section of $\pi$ defined by the projection of $\mathcal O(-1)_{\PP^1}\bigoplus \mathcal O_{\PP^1}^r$ onto the first factor.
\end{Lemma}

\begin{proof}
We prove this by induction on $r$. 

First suppose that $r=0$. Then $\pi$ is an isomorphism, and $X=C$.
$$
\mathcal O_X(1)\otimes \pi^*\mathcal O_{\PP^1}(1)\cong \pi_*\mathcal O_X(1)\otimes \mathcal O_{\PP^1}(1)\cong \mathcal O_{\PP^1}(-1)\otimes \mathcal O_{\PP^1}(1)\cong \mathcal O_{\PP^1},
$$
from which the statement of the lemma follows.

Now suppose that $r>0$ and the statement of the lemma is true for $r-1$. Let $V_0$ be the $\PP^1$-subbundle of $X$ corresponding to projection onto the first $r-1$
factors,
\begin{equation}\label{eq51}
 0\rightarrow \mathcal O_{\PP^1}\rightarrow \mathcal O_{\PP^1}(-1)\bigoplus \mathcal O_{\PP^1}^r \rightarrow \mathcal O_{\PP^1}(-1)\bigoplus \mathcal O_{\PP^1}^{r-1}\rightarrow 0.
 \end{equation}
 Apply $\pi_*$ to the exact sequence 
 $$
 0\rightarrow \mathcal O_X(1)\otimes \mathcal O_X(-V_0)\rightarrow \mathcal O_X(1)\rightarrow \mathcal O_{V_0}(1)\rightarrow 0
 $$
 to obtain the exact sequence (\ref{eq51}), from which we see that $\mathcal O_X(V_0)\cong \mathcal O_X(1)$ and $V_0\cong \PP(\mathcal O_{\PP^1}(-1)\bigoplus \mathcal O_{\PP^1}^{r-1})$ with $\mathcal O_X(V_0)\otimes \mathcal O_{V_0}\cong  \mathcal O_{V_0}(1)$. Let $F$ be the fiber over a point in $\PP^1$ by $\pi$. We have that $\mathcal O_X(1)\otimes \pi^*\mathcal O_{\PP^1}\cong \mathcal O_X(V_0+F)$.
 Apply $\pi_*$ to
 $$
 0\rightarrow \mathcal O_X(F) \rightarrow \mathcal O_X(V_0+F)\rightarrow \mathcal O_X(V_0+F)\otimes \mathcal O_{V_0}\rightarrow 0
 $$
 to get
 $$
 0\rightarrow \mathcal O_{\PP^1}(1)\rightarrow \mathcal O_{\PP^1}\bigoplus \mathcal O_{\PP^1}(1)^r\rightarrow \pi_*(\mathcal O_{V_0}(1)\otimes \pi^*\mathcal
 O_{\PP^1}(1))\rightarrow 0.
 $$
 Now take global sections to obtain that the restriction map 
 $$
 \Gamma(X,\mathcal O_X(V_0+F))\rightarrow \Gamma(V_0, \mathcal O_{V_0}(1)\otimes \pi^*\mathcal
 O_{\PP^1}(1)) 
 $$
 is a surjection. In particular, by the induction statement, $V_0$ contains no base points of $\Lambda=|\Gamma(X,\mathcal O_X(V_0+F))|$. 
 Since any two fibers $F$ over points of $\PP^1$ are linearly equivalent, $\Lambda$ is base point free.

 Suppose that $\gamma$ is a curve of $X$ which is not contained in $V_0$. If $\pi(\gamma)=\PP^1$ then $(\gamma\cdot F)>0$ and $(\gamma\cdot V_0)\ge 0$ so that $\gamma$ is not contracted by $\Lambda$. If $\gamma$ is in a fiber of $F$ then $(\gamma\cdot F)=0$. Let $F\cong \PP^r$ be the fiber of $\pi$ containing $\gamma$. Let $h=F\cdot V_0$, a hyperplane section of $F$. Then
 $(\gamma\cdot V_0)=(\gamma\cdot h)_{F}>0$. Thus $\gamma$ is not contracted by $\Lambda$. By induction on $r$, we have that $C$ is the only curve on $V_0$ which is contracted by $\Lambda$. We have thus proven the induction statement for $r$.
 
 \end{proof}

\begin{Example}\label{Example3} Let $k$ be an algebraically closed field.
Suppose that $s$ is a positive integer and $a_i\in \ZZ_+$ are positive integers for $1\le i\le s$. Suppose that $d>1$.
Then there exists a connected reduced projective scheme $X$ over $k$ which is equidimensional of dimension $d$ with a line bundle $\mathcal L$ on $X$
and a bounded function $\sigma(n)$ such that 
$$
\dim_k\Gamma(X,\mathcal L^n)=\lambda(n)\binom{d+n-1}{d-1}+\sigma(n),
$$
where $\lambda(n)$ is the periodic function
$$
\lambda(n)=|\{i\mid n\equiv 0 (a_i)\}|.
$$
The Kodaira-Iitaka dimension of $L$ is  $\kappa(L)=d-1$. Let $m'=\mbox{LCM}\{a_i\}$.
The limit
$$
\lim_{n\rightarrow \infty}\frac{\dim_k L_n}{n^{d-1}}
$$
exists whenever $n$ is constrained to be in an arithmetic sequence $a+bm'$ (with any fixed $a$). We  have that $\dim_kL_n\ne 0$ for all $n$ if some $a_i=1$, so the
conclusions of Theorem \ref{Theorem5} do not quite hold in this example.
\end{Example}

\begin{proof} Let $E$ be an elliptic curve over $k$. Let $p_0, p_1,\ldots, p_s$ be points on $E$ such that the line bundles $\mathcal O_E(p_i-p_0)$ have order $a_i$. Let $S=E\times_k\PP^{d-1}_k$, and define line bundles $\mathcal L_i=\mathcal O_E(p_i-p_0)\otimes \mathcal O_{\PP^{d-1}}(1)$ on $S$.
The Segre embedding gives  a closed embedding of $S$ in $\PP^{r}$ with $r=3d-1$ Let $\pi:X=\PP(\mathcal O_{\PP^1}(-1)\bigoplus \mathcal O_{\PP^1}^r)
\rightarrow \PP^1$ be the projective bundle, and let $C$ be the section corresponding to the surjection
of $\mathcal O_{\PP^1}(-1)\bigoplus \mathcal O_{\PP^1}^r$ onto the first factor. Let $b_1,\ldots,b_s$ be distinct points of $\PP^1$ and let $F_i$ be the fiber by $\pi$ over $b_i$. Let $S_i$ be an embedding of $S$ in $F_i$. We can if necessary make a translation of $S_i$ so that the point $c_i=C\cdot F_i$ lies
on $S_i$, but is not contained in $p_j\times\PP^{d-1}$ for any $j$. 
We have a line bundle $\mathcal L'$ on the (disjoint) union $T$ of the $S_i$ defined by $\mathcal L'|S_i=\mathcal L_i$.

By Lemma \ref{Lemma52}, there is a morphism $\phi:X\rightarrow Y$ which only contracts the curve $C$. $\phi$ is actually birational and an isomorphism away from $C$, but we do not need to verify this, as we can certainly obtain this  after  replacing $\phi$ with  the Stein factorization of $\phi$. Let $Z=\phi(T)$. 
The birational morphism $T\rightarrow Z$ is an isomorphism away from the points $c_i$, which are not contained on the support of the divisor defining $\mathcal L'$. Thus $\mathcal L'|(T\setminus\phi(C))$ extends naturally to a line bundle $\mathcal L$ on $Z$.

We have a short exact sequence
$$
0\rightarrow \mathcal O_Z\rightarrow \bigoplus_{i=1}^s \mathcal O_{S_i}\rightarrow \mathcal F\rightarrow 0
$$
where $\mathcal F$ has finite support. Tensoring this sequence with $\mathcal L^n$ and taking global sections, we obtain that
$$
0\le \sum_{i=1}^s\dim_k\Gamma(S_i,\mathcal L_i^n)-\dim_k\Gamma(Z,\mathcal L^n) \le \dim_k\mathcal F
$$
for all $n$. Since 
$$
\Gamma(S_i,\mathcal L_i^n)\cong \Gamma(E,\mathcal O_E(n(p_i-p_0)))\otimes_k\Gamma(\PP^{d-1},\mathcal O_{\PP^{d-1}}(n))
$$
by the Kuenneth formula, we obtain the conclusions of the example.

 \end{proof}

\section{Necessary and sufficient conditions for limits to exist on a proper scheme over a field}\label{SecNec}

The nilradical $\mathcal N_X$ of a scheme $X$ is defined in the section on notations and conventions.

\begin{Lemma}\label{Lemmanr90} Suppose that $X$ is a proper scheme over a field $k$ and $L$ is a graded linear series on $X$. Then $\kappa(L)=\kappa(L|X_{red})$.
\end{Lemma}

\begin{proof} Let $\mathcal L$ be a line bundle associated to $X$. We have a commutative diagram
$$
\begin{array}{lllllllll}
&&0&&0&&0&&\\
&&\downarrow&&\downarrow&&\downarrow&&\\
0&\rightarrow & \bigoplus_{n\ge 0}K_n&\rightarrow & \bigoplus_{n\ge 0}L_n&\rightarrow & \bigoplus_{n\ge 0}(L|X_{\rm red})_n&\rightarrow &0\\
&&\downarrow&&\downarrow&&\downarrow&&\\
0&\rightarrow & \bigoplus_{n\ge 0} \Gamma(X,\mathcal L^n\otimes \mathcal N_X)&\rightarrow &\bigoplus_{n\ge 0}\Gamma(X,\mathcal L^n)&\rightarrow &
 \bigoplus_{n\ge 0}\Gamma(X_{\rm red},(\mathcal L|X_{\rm red})^n)
 \end{array}
 $$
 so that $K_n=L_n\cap \Gamma(X,\mathcal L^n\otimes\mathcal N_X)$ for all $n$.
 
 Suppose that $\sigma\in \Gamma(X,\mathcal L^m)$. $X$ is Noetherian, so there exists $r_0=r_0(\sigma)$ such that the closed sets $\mbox{sup}(\sigma^r)=\mbox{sup}(\sigma^{r_0})$ for all $r\ge r_0$. Thus 
 $$
 \begin{array}{l}
 \mbox{$\sigma\in \Gamma(X,\mathcal L^n\otimes\mathcal N_X)$ if and only if}\\
 \mbox{$\sigma_Q$ is torsion in the $\mathcal O_{X,Q}$-algebra $\bigoplus_{n\ge 0}\mathcal L_Q^n$ for all $Q\in X$, if and only if}\\
 \mbox{$\sigma_Q^{r_0}=0$ in $\bigoplus_{n\ge 0}\mathcal L^n_Q$ for all $Q\in X$, if and only if}\\
 \mbox{$\sigma^{r_0}=0$ in $\bigoplus_{n\ge 0}\Gamma(X,\mathcal L^n)$ since $\mathcal L$ is a sheaf.}
 \end{array}
 $$
 Thus $\bigoplus_{n\ge 0}\Gamma(X,\mathcal L^n\otimes \mathcal N_X)$ is the nilradical of $\bigoplus_{n\ge 0}\Gamma(X,\mathcal L^n)$ and 
 so $K$ is the nilradical of $L$.
 
 We have that $\kappa(L|X_{\rm red})\le \kappa(L)$ since any injection of a weighted polynomial ring into $L|X_{\rm red}$ lifts to a graded injection into $L$.
 
 If $A$ is a weighted polynomial ring which injects into $L$, then it intersects $K$ in $(0)$, so there is an induced graded inclusion of $A$ into $L|X_{\rm red}$. Thus $\kappa(L|X_{\rm red})=\kappa(L)$.

\end{proof}

\begin{Theorem}\label{Theorem8} 
Suppose that $X$ is a  proper scheme over a field $k$. Let $\mathcal N_X$ be the nilradical of $X$. Suppose that $L$ is a graded linear series on $X$. Then
\begin{enumerate}
\item[1)]  There exists a positive constant $\gamma$ such that $\dim_kL_n<\gamma n^e$ where 
$$
e=\max\{\kappa(L),\dim \mathcal N_X\}.
$$
\item[2)] Suppose that $\dim \mathcal N_X<\kappa(L)$. Then there exists a positive integer $r$ such that 
$$
\lim_{n\rightarrow \infty}\frac{\dim_kL_{a+nr}}{n^{\kappa(L)}}
$$
exists for any fixed $a\in\NN$. 
\end{enumerate}
\end{Theorem}

\begin{proof} Let $\mathcal L$ be a line bundle associated to $L$, so that $L_n\subset \Gamma(X,\mathcal L^n)$ for all $n$.
Let $K_n$ be the kernel of the surjection $L_n\rightarrow (L|X_{\rm red})_n$.
From the exact sequence
$$
0\rightarrow \mathcal N_X\rightarrow \mathcal O_X\rightarrow \mathcal O_{X_{\rm red}}\rightarrow 0,
$$ 
we see that $K_n\subset \Gamma(X,\mathcal N_X\otimes\mathcal L^n)$ for all $n$. There exists a constant $c$ such that
$$
\dim_k\Gamma(X,\mathcal N_X\otimes\mathcal L^n)<cn^{\dim \mathcal N_X}
$$
for all $n$. By Lemma \ref{Lemmanr90} and Theorem \ref{Theorem18}, 1. holds and there exists a positive integer $r$ such that for any $a$,
$$
\lim_{n\rightarrow \infty}\frac{\dim_k(L|X_{\rm red})_{a+nr}}{n^{\kappa(L)}}
$$
exists. Thus the conclusions of the theorem hold.
\end{proof}

An  example showing that the $r$ of the theorem might have to be strictly larger than the index $m(L)$ is obtained as follows. Let $X_1$ and $X_2$ be two general linear subspaces of dimension $d$  in $\PP^{2d}$. 
They intersect transversally in a rational point $Q$.  Let 
Let $L^i$ be the graded linear series on $X_i$ defined by 
$$
L^1_n=\left\{\begin{array}{ll}
\Gamma(X_1,\mathcal O_{X_1}(n)\otimes \mathcal O_{X_1}(-Q))&\mbox{ if } 2\mid n\\
0&\mbox{ otherwise}\end{array}\right.
$$
and
$$
L^2_n=\left\{\begin{array}{ll}
\Gamma(X_2,\mathcal O_{X_2}(n)\otimes \mathcal O_{X_2}(-Q))&\mbox{ if } 3\mid n\\
0&\mbox{ otherwise}\end{array}\right.
$$
Here $\mathcal O_{X_i}(-Q)$ denotes the ideal sheaf on $X_i$ of the point $Q$. Let $X$ be the reduced scheme whose support is $X_1\cup X_2$.
From the short exact sequence
$$
0\rightarrow \mathcal O_X\rightarrow \mathcal O_{X_1}\bigoplus \mathcal O_{X_1}\rightarrow k(Q)\rightarrow 0,
$$
 we see that there is a graded linear series $L$ on $X$ associated to $\mathcal O_X(1)$ such that
 $L|X_i=L^i$ for $i=1,2$, and $\dim_kL_n= \dim_kL^1_n+\dim_kL^2_n$ for all $n$. Thus
 $$
 \dim_k L_n=\left\{\begin{array}{cl}
 2\binom{d+n}{d}-2&\mbox{ if }n\equiv 0\, (\mbox{mod } 6)\\
 0&\mbox{ if }n\equiv 1\mbox{ or }5\,(\mbox{mod }6)\\
 \binom{d+n}{d}-1&\mbox{ if }n\equiv 2,3\mbox{ or }4\,(\mbox{mod }6).\\
 \end{array}\right.
 $$

In Theorem \ref{TheoremN1}, we give general conditions under which limits do not always exist.

\begin{Theorem}\label{TheoremN1} Suppose that $X$ is a $d$-dimensional projective scheme over a field $k$ with $d>0$.  Let $r=\dim \mathcal N_X$, where $\mathcal N_X$ is the nilradical of $X$. Suppose that $r\ge 0$. Let $s\in \{-\infty\}\cup \NN$ be such that $s\le r$. Then there exists a graded linear series $L$ on $X$ with $\kappa(L)=s$ such that 
$$
\lim_{n\rightarrow\infty} \frac{\dim_k L_n}{n^r}
$$
does not exist, even when $n$ is constrained to lie in any arithmetic sequence.
\end{Theorem}

\begin{Remark} The sequence 
$$
\frac{\dim_kL_n}{n^r}
$$
in Theorem \ref{TheoremN1} must be bounded by Theorem \ref{Theorem8}.
\end{Remark}

\begin{proof} Let $Y$ be an irreducible component of the support of $\mathcal N_X$ which has maximal dimension $r$. Let $S$ be a homogeneous coordinate ring of $X$, which we may assume is saturated, so that the natural graded homomorphism $S\rightarrow \bigoplus_{n\ge 0}\Gamma(X,\mathcal O_X(n))$ is an inclusion. Let $P_Y$ be the homogeneous prime ideal of $Y$ in $S$. There exist homogeneous elements $z_0,\ldots,z_r\in S$ such that if $\overline z_i$ is the image of $z_i$ in $S/P_Y$, then $\overline z_1,\ldots, \overline z_r$ is a homogeneous system of parameters in $S/P_Y$
(by Lemma 1.5.10 and Proposition 1.5.11 \cite{BH}). We can assume that $\deg z_0=1$ since some linear form in $S$ is not in $P_Y$, so it is not a zero divisor $S/P_Y$. We can take $z_0$ to be this form
(If $k$ is infinite, we can take all of the $z_i$ to have degree 1). 
$k[\overline z_0,\ldots,\overline z_r]$ is a weighted polynomial ring (by Theorem 1.5.17 \cite{BH}), so $A:= k[z_0,\ldots,z_r]\cong k[\overline z_0,\ldots,\overline z_d]$ is  a weighted polynomial ring. Let $N$ be the nilradical of $S$. The sheafification of $N$ is $\mathcal N_X$. $P_Y$ is a minimal prime of $N$, so there exists a homogeneous element $x\in N$ such that $\mbox{ann}_S(x)=P_Y$. $N_{P_Y}\ne 0$ in $S_{P_Y}$, so $(P_Y)_{P_Y}\ne 0$. Thus $x\in P_Y$, since otherwise $0=(xP_Y)_{P_Y}=(P_Y)_{P_Y}$. Consider the graded
$k$-subalgebra $B:=A[x]=k[z_0,\ldots,z_r,x]$ of $S$. We have that $x^2=0$. Also, $\mbox{ann}_A(x)=\mbox{ann}_S(x)\cap A=P_Y\cap A=(0)$. Suppose that $ax+b=0$ with $a,b\in A$. Then $xb=0$, whence $b=0$ and thus $a=0$ also. Hence the only relation on $B$ is $x^2=0$. Let $d_i=\deg z_i$, $e=\deg x$. Recall that $d_0=1$. Let $f=\mbox{LCM}\{d_0,\ldots,d_r,e\}$.

First assume  that $r\ge 1$. For $0\le\alpha\le r$, let $M_t^{(\alpha)}$ be the $k$-vector space of homogeneous forms of degree $t$ in the weighted variables $z_0,\ldots,z_{\alpha}$. $\mathcal O_{Z_{\alpha}}(f)$ is an ample line bundle on the weighted projective space $Z_{\alpha}=\mbox{Proj}(k[z_0,\ldots,z_{\alpha}])$ (If $U_i=\mbox{Spec}(k[z_0,\ldots,z_{\alpha}]_{(z_i)})$, then $\mathcal O_{Z_{\alpha}}|U_i=z_i^{\frac{f}{d_i}}\mathcal O_{U_i}$).
$$
\dim_kM_{nf}^{(\alpha)}=\dim_k\Gamma(Z_{\alpha},\mathcal O_{Z_{\alpha}}(nf))
$$
is thus the value of a polynomial $Q_{\alpha}(n)$ in $n$ of degree $\alpha$ for $n\gg0$. Write
$$
Q_{\alpha}(n)=c_{\alpha}n^{\alpha}+\mbox{ lower order terms}.
$$

Suppose that $s\ge 0$ (and $r\ge 1$). 
Let $L_0=k$. For $n\ge 1$, let 
$$
L_n=z_0^{nf}M_{nf}^s+xz_0^{(n-\sigma(n))f-e}M^r_{(n+\sigma(n))f}\subset  B_{2nf}\subset S_{2nf}\subset \Gamma(X,\mathcal O_X(2nf))
$$
where $\sigma(n)$ is the function of (\ref{eqsigma}).
$L_mL_n\subset L_{m+n}$ since $\sigma(j)\ge \sigma(i)$ if $j\ge i$. $L=\bigoplus_{n\ge 0}L_n$ is a graded linear series with $\kappa(L)=s$. Since $B$ has $x^2=0$ as its only relation,  we have that
$$
\begin{array}{lll}
\dim_kL_n&=& \dim_k M_{nf}^s+\dim_k M^r_{(n+\sigma(n))f}\\
&=& Q_s(n)+Q_r(n+\sigma(n)),
\end{array}
$$
so that
$$
\lim_{n\rightarrow \infty}\frac{\dim_k L_n}{n^r}= \lim_{n\rightarrow \infty} \left(c_sn^{s-r}+c_r(1+(\frac{\sigma(n)}{n}))^r\right)
$$
which does not exist, even when $n$ is constrained to lie in any arithmetic sequence, since $\lim_{n\rightarrow \infty}\frac{\sigma(n)}{n}$ has this property (as commented after (\ref{eqnr1})).

Suppose that $s=-\infty$ (and $r\ge 1$). Then define the graded linear series $L$ by 
$$
L_n=xz_0^{(n-\sigma(n))f-e}M^r_{(n+\sigma(n))f}.
$$
Then $\kappa(L)=-\infty$. We compute as above that
$$
\lim_{n\rightarrow\infty} \frac{\dim_kL_n}{n^r}=\lim_{n\rightarrow \infty} c_r(1+(\frac{\sigma(n)}{n}))^r
$$
 does not exist, even when $n$ is constrained to lie in any arithmetic sequence.

Now assume that $r=s=0$. Since $\dim \mathcal N_X=0$, we have injections for all $n$,
$$
\Gamma(X,\mathcal N_X)\cong \Gamma(X,\mathcal N_X\otimes \mathcal O_X(n))\rightarrow \Gamma(X,\mathcal O_X(n)).
$$ 
 In this case $Y$ is a closed point, so that $\dim_k\Gamma(X,\mathcal I_Y\otimes\mathcal O_X(en))$ goes to infinity as $n\rightarrow\infty$ (we assume that $d=\dim X>0$). Thus for $g\gg 0$, there exists $h\in \Gamma(X,\mathcal I_Y\otimes\mathcal O_X(eg))$ such that $h\not\in \Gamma(X,\mathcal N_X\otimes\mathcal O_X(eg))$, so $h$ is not nilpotent in $S$. $h\in P_Y$ implies $hx=0$ in $S$.  Define $L_0=k$ and for $n>0$,
 $$
 L_n =\left\{\begin{array}{ll}
 kh^n&\mbox{ if }\tau(n)=0\\
 kh^n+kxz_0^{ng-e}&\mbox{ if }\tau(n)=1,
 \end{array}\right.
 $$
 where $\tau(n)$ is the function of (\ref{eqtau}). $\tau(n)$ has the property that $\tau(n)$ is not eventually constant, even when $n$ is constrained to line in an arithmetic sequence.

$L=\bigoplus_{n\ge 0}L_n$ is a graded linear series on $X$ with $\kappa(L)=0$ such that 
$\lim_{n\rightarrow\infty}\dim_kL_n$ does not exist, even when $n$ is constrained to lie in any arithmetic sequence.

The last case is when $r=0$ and $s=-\infty$. Define $L_0=k$ and for $n>0$,
$$
L_n=\left\{\begin{array}{ll}
0&\mbox{ if }\tau(n)=0\\
kxz_0^{ng-e}&\mbox{ if }\tau(n)=1.
\end{array}
\right.
$$
Then $L=\bigoplus_{n\ge 0}L_n$ is a graded linear series on $X$ with $\kappa(L)=-\infty$ such that 
$\lim_{n\rightarrow\infty}\dim_kL_n$ does not exist, even when $n$ is constrained to lie in any arithmetic sequence.

\end{proof}

\begin{Theorem}\label{TheoremN20} Suppose that $X$ is a projective nonreduced scheme over a field $k$. Suppose that $s\in \NN\cup\{-\infty\}$ satisfies
$s\le\dim \mathcal N_X$. Then there exists a graded linear series $L$ on $X$ with $\kappa(L)=s$ and a constant $\alpha>0$ such that 
$$
\alpha n^{\dim \mathcal N_X}<\dim_k L_{nm}
$$
for all $n\gg 0$.
\end{Theorem}

\begin{proof} Let $r=\dim \mathcal N_X$. When $r\ge 1$ and $s\le r$, this is established in the construction of Theorem \ref{TheoremN1}.
When $r=0$ and $s=0$, the graded linear series $L=k[t]$ (with associated line bundle $\mathcal O_X$) has $\kappa(L)=0$ and $\dim_kL_n=1$ for all $n$, so satisfies the bound.

Suppose that $r=0$ and $s=-\infty$. Then $0\ne \Gamma(X,\mathcal N_X)$ since the support of $\mathcal N_X$ is zero dimensional. Define $L_0=k$ and $L_n=\Gamma(X,\mathcal N_X)$ for $n>0$. Then $L=\bigoplus_{n\ge 0}L_n$ is a graded linear series for $\mathcal O_X$ with $\kappa(L)=-\infty$ which satisfies the bound.

\end{proof}

\begin{Theorem}\label{TheoremN2}
Suppose that $X$ is a $d$-dimensional projective scheme  over a field $k$ with $d>0$. Let $\mathcal N_X$ be the nilradical of $X$. Let $\alpha\in  \NN$. Then the following are equivalent:
\begin{enumerate}
\item[1)] For every graded linear series $L$ on $X$ with $\alpha\le\kappa(L)$, there exists a positive integer $r$ such that 
$$
\lim_{n\rightarrow\infty}\frac{\dim_kL_{a+nr}}{n^{\kappa(L)}}
$$
exists for every positive integer $a$.
\item[2)] For every graded linear series $L$ on $X$ with $\alpha\le \kappa(L)$, there exists an arithmetic sequence $a+nr$ (for fixed $r$ and $a$ depending on $L$) such that 
$$
\lim_{n\rightarrow\infty}\frac{\dim_kL_{a+nr}}{n^{\kappa(L)}}
$$
exists.
\item[3)] The nilradical $\mathcal N_X$ of $X$ satisfies $\dim \mathcal N_X<\alpha$.
\end{enumerate}
\end{Theorem}

\begin{proof} 1) implies 2) is immediate. 2) implies 3) follows from Theorem \ref{TheoremN1}. 3) implies 1) follows from Theorem \ref{Theorem8}.
\end{proof}

\begin{Theorem}\label{Theorem14a} Suppose that $X$ is a proper scheme of dimension $d$ over a  field $k$, such that $\dim \mathcal N_X<d$ and $\mathcal L$ is a  line bundle on $X$.  Then  the limit
$$
\lim_{n\rightarrow \infty}\frac{\dim_k \Gamma(X,\mathcal L^n)}{n^d}
$$
exists. 
\end{Theorem}

\begin{proof} 
We first prove the theorem in the case when $X$ is integral (a variety). 
We may assume that the section ring $L$ of $\mathcal L$ has maximal Kodaira-Iitaka dimension $d$, because the limit is zero otherwise.
There then exists a positive constant $\alpha$ and a positive integer $e$ such that
$$
\dim_k\Gamma(X,\mathcal L^{ne})>\alpha n^d
$$
for all positive integers $n$ by (\ref{eqKI2}). Let $H$ be a hyperplane section of $X$, giving a short exact sequence
$$
0\rightarrow \mathcal O_X(-H)\rightarrow \mathcal O_X\rightarrow \mathcal O_H\rightarrow 0.
$$
Tensoring with $\mathcal L^n$ and taking global sections, we see that $\Gamma(X,\mathcal L^{ne}\otimes \mathcal O_X(-H))\ne 0$ for $n\gg 0$ as
$q(\mathcal L^e\otimes \mathcal O_H)\le \dim(H)=d-1$. Since $H$ is ample, there exists a positive integer $f$ such that $\mathcal L\otimes\mathcal O_X(fH)$
is generated by global sections. Thus
$$
\Gamma(X,\mathcal L^{nef+1})\cong \Gamma(X,(\mathcal L^{nef}\otimes \mathcal O_X(-fH))\otimes (\mathcal L\otimes \mathcal O_X(fH)))\ne 0
$$
for $n\gg 0$. Thus $m(L)=1$.
The theorem  in the case when $X$ is a variety thus follows from Theorem \ref{Theorem5}.

Now assume that $X$ is  reduced. Let $X_1,\ldots, X_s$ be the irreducible components of $X$. Since $X$ is reduced, we have a natural short exact sequence of
$\mathcal O_X$-modules
$$
0\rightarrow \mathcal O_X\rightarrow \bigoplus_{n\ge 0}\mathcal O_{X_i}\rightarrow \mathcal F\rightarrow 0
$$
where $\mathcal F$ has support of dimension $\le d-1$. Tensoring with $\mathcal L^n$, we obtain that
$$
\lim_{n\rightarrow \infty}\frac{\dim_k\Gamma(X,\mathcal L^n)}{n^d}=\sum_{i=1}^s\lim_{n\rightarrow \infty}\frac{\dim_k\Gamma(X_i, \mathcal L^n\otimes\mathcal O_{X_i})}{n^d}
$$
exists, as $\dim_k\Gamma(X,\mathcal F\otimes \mathcal L^n)$ grows at most like $n^{d-1}$.

Let $\overline X=X_{\rm red}$ so that $\mathcal O_{\overline X}=\mathcal O_X/\mathcal N_X$. From the exact sequence
$$
0\rightarrow \mathcal N_X\rightarrow \mathcal O_X\rightarrow \mathcal O_{\overline X}\rightarrow 0
$$
and since the support of $\mathcal N_X$ has dimension less than $d$, we have that
$$
\lim_{n\rightarrow\infty}\frac{\dim_k\Gamma(X,\mathcal L^n)}{n^d}=\lim_{n\rightarrow\infty}\frac{\dim_k\Gamma(\overline X,\mathcal L^n\otimes\mathcal O_{\overline X})}{n^d}
$$
exists.
\end{proof}

\section{Nonreduced zero dimensional schemes}\label{SecZero}

The case when $d=\dim X=0$ is  rather special. In fact, the implication 2) implies ) of Theorem \ref{TheoremN2} does not hold if $d=0$, as follows from
Proposition \ref{TheoremN3} below. There is however a very precise statement about what does happen in zero dimensional schemes, as we show below.

\begin{Proposition}\label{TheoremN3} Suppose that $X$ is a 0-dimensional irreducible but nonreduced $k$-scheme and $L$ is a graded linear series on $X$ with $\kappa(L)=0$. Then there exists a positive integer $r$ such that  
$$
\lim_{n\rightarrow\infty}\dim_kL_{a+nr}
$$
exists for every positive integer $a$.
\end{Proposition}

\begin{proof} With our assumptions, $X=\mbox{Spec}(A)$ where $A$ is a nonreduced Artin local ring, with $\dim_kL<\infty$, and $L$ is a graded $k$-subalgebra of $\Gamma(X,\mathcal O_X)[t]=A[t]$. The condition $\kappa(L)=0$ is equivalent to the statement that there exists $r>0$ such that $L_r$ contains a unit $u$ of $A$. We then have that 
$$
\dim_kL_{m+r}\ge \dim_kL_mL_r\ge \dim_kL_m
$$
for all $m$. Thus for fixed $a$, $\dim_kL_{a+nr}$ must stabilize for large $n$.
\end{proof}

We do not have such good behavior for graded linear series $L$ with $\kappa(L)=\infty$. 

\begin{Proposition}\label{TheoremN6} Suppose that $X$ is a 0-dimensional  nonreduced $k$-scheme. Then there exists a  graded linear series $L$ on $X$ with $\kappa(L)=-\infty$, such that 
$$
\lim_{n\rightarrow\infty}\dim_kL_{n}
$$
does not exist, even when $n$ is constrained to lie in any arithmetic sequence.
\end{Proposition}

\begin{proof}  $X=\mbox{Spec}(A)$ where $A=\bigoplus_{i=1}^sA_i$, with $s$ the number of irreducible components of $X$ and the $A_i$ are  Artin local rings with $\dim_kA<\infty$. Let $m_{A_1}$ be the maximal ideal of $A_1$.  
There exists a number $0<t$ such that $m_{A_1}^t\ne 0$ but $m_{A_1}^{t+1}=0$. Let $\tau(n)$ be the function defined in (\ref{eqtau}).

Define a graded linear series $L^1$  on $\mbox{Spec}(A_1)$ by $L^1_n=m{A_1}^{t+\tau(n)}$. Then $\lim_{n\rightarrow \infty}\dim_k L_n^1$ does not exist, even when $n$ is constrained to  lie in an arithmetic sequence. Extend $L$ to a graded linear series $L$ on $X$ with $\kappa(L)=-\infty$ be setting
$L_n=L_n^1\bigoplus (0)\bigoplus\cdots\bigoplus (0)$.
\end{proof}

It follows that the conclusions of Proposition \ref{TheoremN3} do not hold in nonreduced 0-dimensional schemes which are not irreducible.

\begin{Proposition}\label{TheoremN7}  Suppose that $X$ is a 0-dimensional {\it nonirreducible} and nonreduced $k$-scheme. Then there exists a  graded linear series $L$ on $X$ with $\kappa(L)=0$, such that 
\begin{equation}\label{eqN10}
\lim_{n\rightarrow\infty}\dim_kL_{n}
\end{equation}
does not exist, even when $n$ is constrained in any arithmetic sequence.
\end{Proposition}

\begin{proof} $X=\mbox{Spec}(A)$ where $A=A_1\bigoplus A_2$, with $A_1$ an Artin local ring and $A_2$  an Artin ring.
 A graded linear series $L$ on $X$ is a graded $k$-subalgebra of $A[t]$. Let $L^2$ be a graded linear series on $\mbox{Spec}(A_2)$ with $\kappa(L^2)=-\infty$, such that the conclusions of Proposition \ref{TheoremN6} hold. Then the linear series $L$ on $X$ defined by 
$$
L_n=A_1\bigoplus (L^2)_n
$$
has $\kappa(L)=0$, but 
$$
\lim_{n\rightarrow\infty}\dim_kL_{n}=1+\lim_{n\rightarrow \infty}\dim_k L^2_n
$$
does not exist, even when $n$ is constrained to lie in any arithmetic sequence.  

\end{proof}

In particular, the conclusions of Theorem \ref{TheoremN2} are true for 0-dimensional projective $k$-schemes which are not irreducible.

\section{Examples with Kodaira-Iitaka dimension $-\infty$}\label{Infinity}
It is much easier to construct perverse examples with Kodaira-Iitaka dimension $-\infty$, since the condition $L_mL_n\subset L_{m+n}$ can be trivial
in this case. If $X$ is a reduced variety, and $L$ is a graded linear series on $X$, then it follows from Corollary \ref{Cor72} that there is an upper bound 
$\dim_kL_n<\beta n^{q(L)}$
for all $n$. However, for nonreduced varieties of dimension $d$, we only have the upper bound $\dim_kL_n<\gamma n^d$ of (\ref{eqKI4}).
Here is an example with $\kappa(L)=-\infty$ and maximal growth of order $n^d$.

\begin{Example}\label{Example2} Let $k$ be a field, and let $X$ be the one dimensional projective non reduced  $k$-scheme consisting of a double line in $\PP^2_k$. Let $T$ be a subset of the positive integers.
 There exists a graded  linear series
$L$  for $\mathcal O_X(2)$ such that 
$$
\dim_kL_n= \left\{\begin{array}{ll}
n+1&\mbox{ if }n\in T\\
0&\mbox{ if }n\not\in T
\end{array}\right.
$$

In the example, 
we have that $\kappa(L)=-\infty$, but $\dim_kL_n$ is $O(n)=O(n^{\dim X})$. 
\end{Example}

\begin{proof}
 We can choose homogeneous coordinates coordinates on $\PP^2_k$ so that $X=\mbox{Proj}(S)$, where $S=k[x_0,x_1,x_2]/(x_1^2)$. Let $\overline x_i$ be the classes of $x_i$ in $S$, so that $S=k[\overline x_0,\overline x_1,\overline x_2]$.
Define a graded linear series $L$ for $\mathcal O_X(2)$ by
defining $L_n$ to be the $k$-subspace of $\Gamma(X,\mathcal O_X(2n))$ spanned by $\{\overline x_1\overline x_0^i\overline x_2^j\mid i+j=n\}$
if $n\in T$ and $L_n=0$ if $n\not\in T$. Then
$$
\dim_kL_n= \left\{\begin{array}{ll}
n+1&\mbox{ if }n\in T\\
0&\mbox{ if }n\not\in T
\end{array}\right.
$$
\end{proof}

We modify the above example a little bit to find another  example with interesting growth.

\begin{Theorem}\label{Theorem21} Let $k$ be a field, and let $X$ be the one dimensional projective non reduced  $k$-scheme consisting of a double line in $\PP^2_k$. Let $T$ be any infinite subset of the positive integers $\ZZ_+$ such that $\ZZ_+\setminus T$ is also infinite.
 There exists a graded  linear series
$L$  for $\mathcal O_X(2)$ such that 
$$
\dim_kL_n= \left\{\begin{array}{ll}
\lceil \log(n)\rceil &\mbox{ if }n\in T\\
\lceil \frac{\log(n)}{2}\rceil&\mbox{ if }n\not\in T
\end{array}\right.
$$
In this example we have $\kappa(L)=-\infty$.
\end{Theorem}

\begin{proof}
 We can choose homogeneous coordinates coordinates on $\PP^2_k$ so that $X=\mbox{Proj}(S)$, where $S=k[x_0,x_1,x_2]/(x_1^2)$. Let $\overline x_i$ be the classes of $x_i$ in $S$, so that $S=k[\overline x_0,\overline x_1,\overline x_2]$.
Define
$$
\lambda(n)=\left\{\begin{array}{ll}
\lceil \log(n)\rceil &\mbox{ if }n\in T\\
\lceil \frac{\log(n)}{2}\rceil&\mbox{ if }n\in \ZZ_+\setminus T.
\end{array}\right.
$$ 
Define a graded linear series $L$ for $\mathcal O_X(1)$ by
defining $L_n$ to be the $k$-subspace of $\Gamma(X,\mathcal O_X(n))$ spanned by 
$$
\overline x_0^{n-1}\overline x_1,\overline x_0^{n-2}\overline x_1\overline x_2,\ldots, \overline x_0^{n-\lambda(n)}\overline x_1\overline x_2^{\lambda(n)-1}.
$$
Then $L_n$ has the desired property.
\end{proof}

The following is an example of a line bundle on a non reduced scheme for which there is interesting growth. The characteristic $p>0$ plays a role in the construction.

\begin{Example}\label{Example16} Suppose that  $d\ge 1$. There exists an irreducible but  nonreduced projective variety $Z$ of dimension $d$ over a field of  positive characteristic $p$,   and a line bundle $\mathcal N$ on $Z$, whose Kodaira-Iitaka dimension is  $-\infty$, such that 
$$
\dim_k\Gamma(Z,\mathcal N^{n})=
\left\{\begin{array}{ll}
 \binom{d+n-1}{d-1}&\mbox{ if $n$ is  a power of $p$}\\
 \\
0&\mbox{ otherwise}
\end{array}\right.
$$

\end{Example}
In particular,  given a positive integer $r$, there exists at least one integer   $a$ with $0\le a<r$ such that the limit
$$
\lim_{n\rightarrow \infty}\frac{\dim_k\Gamma(Z,\mathcal N^n)}{n^{d-1}}
$$
does not exist when $n$ is constrained to lie   in the arithmetic sequence $a+br$. 

\begin{proof}

Suppose that $p$ is a prime number such that $p\equiv 2\,\, (3)$.
In Section 6 of \cite{CS}, a projective genus 2 curve $C$ over an algebraic function field $k$ of characteristic $p$ is constructed, which has a $k$-rational point $Q$ and a degree zero line bundle $\mathcal L$ with the properties that
$$
\dim_k \Gamma(C,\mathcal L^n\otimes \mathcal O_C(Q))
=\left\{\begin{array}{ll}
1&\mbox{ if $n$ is  a power of $p$}\\
0&\mbox{ otherwise}
\end{array}\right.
$$
and
\begin{equation}\label{eqG32}
\Gamma(C,\mathcal L^n)=0\mbox{ for all }n.
\end{equation}

Let $\mathcal E=\mathcal O_C(Q)\bigoplus \mathcal O_C$. Let $S=\PP(\mathcal E)$ with natural projection $\pi:S\rightarrow C$, a ruled surface over $C$. Let $C_0$ be the section of $\pi$ corresponding to the
 surjection onto the second factor $\mathcal E\rightarrow \mathcal O_C\rightarrow 0$. By Proposition V.2.6 \cite{H}, we have that
$\mathcal O_S(-C_0)\otimes_{\mathcal O_S}\mathcal O_{C_0}\cong \mathcal O_C(Q)$. Let $X$ be the nonreduced subscheme $2C_0$ of $S$. We have a short exact sequence
$$
0\rightarrow \mathcal O_{C}(Q)\rightarrow \mathcal O_{X}\rightarrow \mathcal O_C\rightarrow 0.
$$
 Let $\mathcal M=\pi^*(\mathcal L)\otimes_{\mathcal O_S}\mathcal O_X$. Then
we have short exact sequences
\begin{equation}\label{eqG31}
0\rightarrow \mathcal L^n\otimes_{\mathcal O_C}\mathcal O_C(Q)\rightarrow \mathcal M^n\rightarrow \mathcal L^n\rightarrow  0.
\end{equation}

 By (\ref{eqG31}) and (\ref{eqG32}), we have that

$$
\begin{array}{lll}
\dim_k\Gamma(X,\mathcal M^n)&=&\dim_k\Gamma(C,\mathcal L^n\otimes \mathcal O_C(Q))\\
\\
&=&\left\{
\begin{array}{ll}
1&\mbox{ if $n$ is  a power of $p$}\\
0&\mbox{ otherwise}
\end{array}\right.\\
\end{array}
$$

Now let $Z=X\times \PP_k^{d-1}$ and $\mathcal N=\mathcal M\otimes \mathcal O_{\PP}(1)$. By the Kuenneth formula, we have that
$$
\Gamma(Z,\mathcal N^n)=\Gamma(X,\mathcal M^n)\otimes_k\Gamma(\PP^{d-1},\mathcal O_{\PP}(n))
$$
from which the conclusions of the example follow.
\end{proof}


\begin{thebibliography}{1000000000}
\bibitem{BC} S. Boucksom and H. Chen, Okounkov bodies of filtered linear series, Compos. Math. 147 (2011), 1205 - 1229.
\bibitem{BFJ} S. Boucksom, C. Favre, M. Jonsson, Differentiability of volumes of divisors and a problem of Teissier, Journal of Algebraic Geometry 18 (2009), 279 - 308.
\bibitem{Br} M. Brodmann, Asymptotic stablity of $\mbox{Ass}(M/I^nM)$, Proc. Amer. Math. Soc 74 (1979), 16 - 18.
\bibitem{BH} W. Bruns and J. Herzog, Cohen-Macaulay rings, Cambridge University Press, 1993.
\bibitem{C} S.D. Cutkosky, Asymptotic growth of saturated powers and epsilon multiplicity, Math. Res. Lett. 18 (2011), 93 - 106,
\bibitem{C1} S.D. Cutkosky, Multiplicities associated to graded families of ideals, Algebra and Number Theory 7 (2013), 2059 - 2083.
\bibitem{C2} S.D. Cutkosky, Teissier's problem on inequalities of nef divisors over an arbitrary field, to appear in Journal of Algebra and its Applications, Memorial volume in honor of Professor Shreeram S. Abhyankar, arXiv:1304.1218.
\bibitem{C3} S.D. Cutkosky, Asymptotic Multiplicities, preprint, arXiv:1311.1432.
\bibitem{CHS} S.D. Cutkosky, J. Herzog and H. Srinivasan, Asymptotic growth of algebras associated to powers of ideals, Math. Proc. Camb. Philos. Soc. 148 (2010), 55 - 72.
\bibitem{CDK} S.D. Cutkosky, K. Dalili and O. Kashcheyeva, Growth of rank 1 valuations, Comm. Algebra 38 (2010), 2768 - 2789.
\bibitem{CHST} S.D. Cutkosky, T. Ha, H. Srinivasan and E. Theodorescu, Asymptotic behavior of length of local cohomology,
Canad. J. Math. 57 (2005), 1178 -1192.
\bibitem{CS} S.D. Cutkosky and V. Srinivas, On a problem of Zariski on dimensions of linear systems, Ann. Math. 137 (1993), 531 - 559.
\bibitem{DJ} A. J. de Jong, Smoothness, semi-stability and alterations, Publications Mathematiques I.H.E.S. 83 (1996), 51 - 93.

\bibitem{ELMNP1} L. Ein, R. Lazarsfeld, M. Musta\c{t}\u{a}, M. Nakamaye, M. Popa, Asymptotic invariants of line bundles, Pure Appl. Math. Q. 1 (2005), 379-403.
\bibitem{ELS} L. Ein, R. Lazarsfeld and K. Smith, Uniform Approximation of Abhyankar valuation ideals in smooth function fields,
Amer. J. Math. 125 (2003), 409 - 440. 
\bibitem{Fuj} T. Fujita, Approximating Zariski decomposition of big line bundles, Kodai Math. J. 17 (1994), 1-3.
\bibitem{Ful} M. Fulger, Local volumes on normal algebraic varieties, arXiv:1105.298
\bibitem{EGAIV} A. Grothendieck, and A. Dieudonn\'e, El\'ements de g\'eom\'etrie alg\'ebrique IV, vol. 2, Publ. Math. IHES 24 (1965).
\bibitem{H} R. Hartshorne, Algebraic Geometry, Springer Verlag 1977.

\bibitem{HPV} J. Herzog, T. Puthenpurakal, J. Verma, Hilbert polynomials and powers of ideals, Math. Proc. Cambridge Math. Soc. 145 (2008), 623 - 642.
\bibitem{I} S. Iitaka, Algebraic Geometry, Springer Verlag, 1982.
\bibitem{KK} K. Kaveh and G. Khovanskii, Newton-Okounkov bodies, semigroups of integral points, graded algebras and intersection theory, to appear in Annals of Math., arXiv:0904.3350v3.
\bibitem{KT} G. Kemper and N.V. Trung, Krull Dimension and Monomial Orders, J. Algebra 399 (2014), 782 - 800.
\bibitem{K} S. Kleiman, Private Communication
\bibitem{KUV} S. Kleiman, B. Ulrich and J. Validashti, Specialization of integral dependence, in preparation.
\bibitem{KKu} H. Knaf and F.-V. Kuhlmann, Abhyankar places admit local uniformization in any characteristic, Ann. Scient. Ec. Norm. Sup 30 (2005), 833 - 846.
\bibitem{La} R. Lazarsfeld, Positivity in Algebraic Geometry, I and II, Ergebnisse der Mathematik und ihrer Grenzgebiete, vols 48 and 49, Springer Verlag, Berlin 2004.
\bibitem{LM} R. Lazarsfeld and M. Musta\c{t}\u{a}, Convex bodies associated to linear series, Ann. Sci. Ec. Norm. Super 42 (2009) 783 - 835.
\bibitem{Ma} H. Matsumura, Commutative Algebra, 2nd edition, Benjamin/Cummings (1980).
\bibitem{Ma2} H. Matsumura, Commutative Ring Theory, Cambridge Univ. Press, 1986.
\bibitem{Mus} M. Musta\c{t}\u{a}, On multiplicities of graded sequence of ideals, J. Algebra 256 (2002), 229-249.
\bibitem{N} M. Nagata, Local Rings, Wiley, 1962.
\bibitem{Ok} A. Okounkov, Why would multiplicities be log-concave?, in The orbit method in geometry and physics, Progr. Math. 213, 2003, 329-347.
\bibitem{R3}  D. Rees, A note on analytically unramified local rings, J. London Math. Soc. 36 (1961), 24 - 28.
\bibitem{Se} J.P. Serre, Alg\`ebre locale. Multiplicit\'es. LNM 11, Springer, 1965.
\bibitem{S} I. Swanson, Powers of ideals: primary decomposition, Artin-Rees lemma and regularity,
Math. Annalen 307 (1997), 299 - 313.
\bibitem{SH} I. Swanson and C. Huneke, Integral closure of ideals, rings and modules, Cambridge Univ. Press, 200.
\bibitem{T} S. Takagi, Fujita's approximation theorem in positive characteristics, J. Math. Kyoto Univ, 47 (2007), 179 - 202.
\bibitem{UV} B. Ulrich and J. Validashti, Numerical criteria for integral dependence,  Math. Proc. Camb. Phil. Soc. 151 (2011), 95 - 102.
\bibitem{ZS2} O. Zariski and P. Samuel, Commutative Algebra Vol II, Van Nostrand (1960).
\bibitem{Z} O. Zariski, The concept of a simple point on an abstract algebraic variety, Trans. Amer. Math. Soc. 62 (1947), 1- 52.
\end{thebibliography}
\end{document}